\documentclass[11pt,reqno]{amsart}

\usepackage{mathrsfs}

\usepackage{hyperref}

\usepackage{graphicx}

\usepackage{xcolor}

\usepackage{enumerate}
\usepackage{bm}

\newtheorem{theorem}{Theorem}[section]

\newtheorem{proposition}[theorem]{Proposition}
\newtheorem{corollary}[theorem]{Corollary}

\theoremstyle{definition}
\newtheorem{definition}[theorem]{Definition}

\theoremstyle{remark}
\newtheorem{remark}[theorem]{Remark}

\numberwithin{equation}{section}

\title[Weak solutions to  Hessian type equations  ]{Weak solutions to Hessian type equations on compact Riemannian manifolds}

\author{Zhenan Sui}

\address{Institute for Advanced Study in Mathematics of HIT, Harbin Institute of Technology, Harbin, China}
\email{sui.4@osu.edu}

\author{Wei Sun}

\address{Institute of Mathematical Sciences, ShanghaiTech University, Shanghai, China}
\email{sunwei@shanghaitech.edu.cn}

\begin{document}
\setlength{\baselineskip}{1.2\baselineskip}

\begin{abstract}

In this paper, we shall study the existence of weak solutions to Hessian type equations on  compact Riemannian manifolds without boundary.

\end{abstract}



\maketitle


\section {Introduction}
\label{introduction}

Let $(M,g)$ be a smooth compact Riemannian manifold without boundary of dimension $n \geq 2$, and $\chi$ be a smooth twice-covariant symmetric tensor field. In this paper, we shall study the weak solutions to the following equation:
\begin{equation}
\label{equation-hessian}
	\mathfrak{F} \left(\chi + \nabla^2 \varphi\right):= \mathfrak{f} \left(\bm{\lambda} (\chi + \nabla^2 \varphi)\right) = e^{f} , \qquad \sup_M \varphi = 0 ,
\end{equation}
where $\nabla^2 \varphi$ denotes the Hessian of $\varphi$ if $\varphi \in C^2 (M)$, $\bm{\lambda} (\chi + \nabla^2 \varphi) = (\lambda_1, \ldots, \lambda_n)$ is the eigenvalue set of $\chi + \nabla^2 \varphi$ with respect to $g$, and $e^f$ is a prescribed function on $M$.

Classical solutions to this type of equations on compact Riemannian manifolds have been studied extensively by Li~\cite{Li1990}, Urbas~\cite{Urbas}, Guan~\cite{Guan2014, Guan2023} and Sz\`ekelyhidi~\cite{Szekelyhidi}. To our best of knowledge, we do not find any work on weak solutions to fully nonlinear equations in the form \eqref{equation-hessian} on Riemannian manifolds.
In this paper, we are interested in the
existence of
weak solutions to equation \eqref{equation-hessian} on compact Riemannian manifolds without boundary, by establishing $L^{\infty}$-estimate and stability estimate. This study is inspired by the
works of ~\cite{GPT2021, GP2022, SuiSun2023, Sun202305MA, Sun202305} for complex Hessian type equations on compact K\"ahler manifolds or Hermitian manifolds, where the estimates and existence results are discussed under very weak regularity of $e^f$.

Similar to the assumptions in Guo-Phong-Tong \cite{GPT2021} and Guo-Phong \cite{GP2022}, $\mathfrak{f}$ is assumed to satisfy the following structural conditions throughout this paper:
\begin{enumerate}
\item $\mathfrak{f} (\bm{\lambda})$ is a function in $C(\bar \Gamma) \cap C^\infty (\Gamma) $ which vanishes on the boundary $\partial \Gamma$,
where $\Gamma \subset \Gamma^1 := \{\bm{\lambda} \in \mathbb{R}^n \,| \, \lambda_1 + \cdots + \lambda_n > 0  \}$ is an open symmetric convex domain
;

\item $\mathfrak{f} (\bm{\lambda})$ is invariant under permutations of the components of $\bm{\lambda}$;

\item $\mathfrak{f} (\bm{\lambda})$ is elliptic, that is,
\begin{equation}
\frac{\partial \mathfrak{f}}{\partial \lambda_i} (\bm{\lambda}) > 0 , \qquad \text {for any } 1 \leq i \leq n \text{ and } \bm{\lambda} \in \Gamma;
\end{equation}

\item there is a constant $c > 0$ such that for any $\bm{\lambda} \in \Gamma$,
\begin{equation}
\label{condition-1-3}
	\prod^n_{i = 1} \frac{\partial \mathfrak{f}}{\partial \lambda_i} (\bm{\lambda}) \geq c ;
\end{equation}

\item $\mathfrak{f} (\bm{\lambda})$ is concave.

\end{enumerate}
These conditions are quite fundamental and widely adopted in the study of fully nonlinear elliptic equations.
We shall discuss these conditions in more detail in Section~\ref{preliminary}.
Without loss of generality, we may further assume that  there is a constant $\sigma > 0$ such that $\bm{\lambda} (\chi - \sigma g) \in \Gamma$ and $\bm{1} := (1,\cdots,1) \in \Gamma$.

In this paper, we shall study the weak solutions to Equation~\eqref{equation-hessian} with the help of $L^{\infty}$ estimate and stability estimate.
Similar estimates have been derived in \cite{GPT2021, GP2022, SuiSun2023, Sun202305MA, Sun202305}, where the authors compared complex Hessian type equations with complex Monge-Amp\`ere equations followed by De Giorgi iteration argument under very weak regularity of prescribed functions on the right hand side of the equations.
Different from the complex case, a difficulty for real equation \eqref{equation-hessian} is that we have to deal with the gradient term in $\nabla^2 \varphi$ within local coordinate charts.
We shall directly adopt a version of Alexandrov-Bakelman-Pucci maximum principle due to Sz\'ekelyhidi \cite{Szekelyhidi}.
Indeed, our argument in this paper applies to similar problems on compact almost complex manifolds.

Our first  result is on  $L^{\infty}$-estimate for Equation \eqref{equation-hessian}.
\begin{theorem}
\label{Theorem-1-1}
Suppose that $f$ is continuous and $\int_M e^{nf} (1 + n |f|)^p d vol < +\infty $ with $p > 0$.
Let $\varphi$ be a $C^2$ solution to Equation~\eqref{equation-hessian}.
Then there is a constant $C$ depending on $\int_M e^{nf} (1 + n |f|)^p d vol $ and geometric data such that $- \varphi < C$.
In particular, there is a constant $C$ depending on $\Vert e^f \Vert_{L^{nq}}$ with $q > 1$ and geometric data such that $- \varphi < C$.
\end{theorem}

The second result is a version of stability estimate, which can help us to construct weak solutions.
\begin{theorem}
\label{theorem-stability-1}
Assume that there are $C^2$ functions $\varphi_1$ and $\varphi_2$ such that
\begin{equation*}
\mathfrak{F} (\chi + \nabla^2 \varphi_1) = e^f , \qquad \sup_M \varphi_1 = 0 ,
\end{equation*}
and
\begin{equation*}
\mathfrak{F} (\chi + \nabla^2 \varphi_2) = e^{\tilde f} , \qquad \sup_M \varphi_2 = 0 .
\end{equation*}
Suppose that for $q > 1$, there is a constant $K > 0$ 
so that
\begin{equation*}
\Vert e^{nf} \Vert_{L^q} \leq K
.
\end{equation*}
Then for any $p > 0$, there exists a constant $C > 0$  such that
\begin{equation*}
	 \sup_M (\varphi_2 -\varphi_1 ) \leq C \left\Vert (\varphi_2 -\varphi_1 )^+ \right\Vert_{L^p}^{\frac{p (q - 1)}{n q + p (q - 1)}} .
\end{equation*}

\end{theorem}

Since $M$ is a manifold without boundary, Equation~\eqref{equation-hessian} might not admit any solution even if $f$ is smooth, in views of the maximum principle.
Actually, we plan to find a solution pair $(\varphi, b)$ to the following equation
\begin{equation}
\label{equation-true}
	\mathfrak{F} \left(\chi + \nabla^2 \varphi\right)
	= e^b e^{f} , \qquad \sup_M \varphi = 0 ,
\end{equation}
where $b$ is a real number to be determined.
The main result of this paper is the following existence theorem.
\begin{theorem}
\label{main-theorem}
Let $(M, g)$ be a smooth compact Riemannian manifold without boundary of dimension $n \geq 2$, and $\chi$ be a smooth twice-covariant symmetric tensor field.
Suppose that $e^f \in L^{nq} (M)$ with $q > 1$ and $\int_M e^f d vol > 0$.
Then Equation~\eqref{equation-true} admits a weak $C^0$ solution pair $(\varphi,b)$ with respect to $L^{nq}$ norm,
where the constant $b$ is uniquely determined by $e^f$.

In particular,  Equation~\eqref{equation-true} admits a viscosity solution pair $(\varphi,b)$, if $e^f \in C (M) $ and $e^f \not\equiv 0$.

\end{theorem}

The definition of weak $C^0$ solution first appeared in Chen~\cite{Chen2000}. We note that the notation of $e^f$ is just to emphasize that this is a nonnegative function in Equation~\eqref{equation-hessian}.
If $e^f \in C(M)$ is not identically $0$, the weak $C^0$ solution with respect to $C^0$ norm in Theorem \ref{main-theorem} turns out to be a viscosity solution to Equation~\eqref{equation-hessian}.
Different from complex equations, real equations such as Equation \eqref{equation-hessian} do not carry the property of invariant integral. Therefore, we need to prove the uniqueness and the bound of $e^b$.

If we further impose some extra conditions, the next result shows the existence of Lipschitz solution.

\begin{theorem} \label{Lipschitz solution}
Under the assumptions of Theorem~\ref{main-theorem}, assume in addition that
\begin{equation}
\label{condition-5-32}
 t \bm{\lambda} \in \Gamma , \qquad \forall \bm{\lambda} \in \Gamma \text{ and } t \geq 1,
\end{equation}
and $e^f$ is locally Lipschitz in an open set $U \subset M$. Then there is a viscosity solution pair $(\varphi,b)$ which is also locally Lipschitz on $U$.
\end{theorem}

This result is based on an interior gradient estimate, which indeed works for more general Hessian type equations on Riemannian manifolds. More details are provided in Section \ref{Interior gradient estimate}.

This paper is organized as follows.
The proof of Theorem \ref{main-theorem} is based on $L^{\infty}$ estimate in Section~\ref{L-estimate} and stability estimate in Section~\ref{S-estimate}. The definitions and existence of weak $C^0$ solution and viscosity solution will be given in Section~\ref{weak-solution}. The interior gradient estimate will be derived in Section \ref{Interior gradient estimate}.

\medskip
\section{Preliminaries}
\label{preliminary}

In this section, we shall state some elementary notations  and formulae.

By Condition (1) and (3),
we know that for any $\bm{\lambda} \in \Gamma$,
\begin{equation}
\mathfrak{f} (\bm{\lambda}) > 0 
\qquad\textup{ and }\qquad
\bm{\lambda} + \Gamma^n \subset \Gamma
,
\end{equation}
where $\Gamma^n := \{\bm{\lambda} \in \mathbb{R}^n | \lambda_1 > 0, \cdots, \lambda_n > 0\}$.

By Condition (2), (4) and (5), we have
\begin{equation}
\frac{\partial\mathfrak{f}}{\partial\lambda_1} (t\bm{1}) = \cdots = \frac{\partial\mathfrak{f}}{\partial\lambda_n} (t\bm{1}) \geq c^{\frac{1}{n}} .
\end{equation}
Then by concavity again
\begin{equation}
\label{inequality-2-3}
	\frac{\partial \mathfrak{f}}{\partial \lambda_1} (\bm{1}) \left(S_1 (\bm{\lambda}) - n \right)
	=
	\sum_i \frac{\partial\mathfrak{f} }{\partial\lambda_i} (\bm{1})  (\lambda_i - 1)
	\geq
	\mathfrak{f} (\bm{\lambda}) - \mathfrak{f} (\bm{1})
	.
\end{equation}
Moreover,
\begin{equation}
	\frac{\partial \mathfrak{f}}{\partial \lambda_i} (s \bm{1}) \geq \frac{\partial\mathfrak{f}}{\partial \lambda_i} (t\bm{1}) \geq c^{\frac{1}{n}},
\end{equation}
for any $t  > s > 0$ if $s \bm{1} \in \Gamma$.
In particular, there exists a positive limit for $\frac{\partial\mathfrak{f}}{\partial\lambda_i} (t \bm{1})$.
Thus we have
\begin{equation}
\lim_{t \to +\infty} \mathfrak{f} (t\bm{1}) = +\infty .
\end{equation}

Condition (4) 
implies that
\begin{equation}
\label{inequality-2-6}
\begin{aligned}
\mathfrak{f} \left(\bm{\lambda}\right)
\geq
\mathfrak{f} \left(\bm{\lambda}\right) - \mathfrak{f}\left(\bm{\lambda} - \bm{\mu}\right)
&\geq
\sum_i \frac{\partial\mathfrak{f}}{\partial\lambda_i} (\bm{\lambda}) \mu_i \\
&\geq n \left(\prod_i \frac{\partial\mathfrak{f}}{\partial\lambda_i} (\bm{\lambda})\right)^{\frac{1}{n}} \left(\prod_i \mu_i\right)^{\frac{1}{n}}
\\
&\geq n \left( c \prod_i \mu_i\right)^{\frac{1}{n}}
\end{aligned}
\end{equation}
for $\bm{\lambda} - \bm{\mu}\in \Gamma$ and $\bm{\mu} \in \Gamma^n$.

At each point $\bm{x} \in M$, define $\bm{\lambda}' = \bm{\lambda} (\chi)$. Then for any continuous function $e^f$, the set
\begin{equation}
\left\{\bm{\lambda} \in \Gamma \,|\, \mathfrak{f}(\bm{\lambda}) = e^f (\bm{x} ) \text{ and } \bm{\lambda} - \bm{\lambda}' \in \Gamma^n\right\}
\end{equation}
is uniformly bounded, as
\begin{equation}
\label{inequality-2-14}
\begin{aligned}
	e^f - \mathfrak{f} (\bm{\lambda}' - \sigma \bm{1})
	\geq
	\sum_i \frac{\partial\mathfrak{f}}{\partial\lambda_i} (\bm{\lambda}) ({\lambda}_i - {\lambda}'_i + \sigma  )
	\geq
	n \left(c \sigma^{n - 1} \left(\max (\lambda_i - \lambda'_i) + \sigma\right)\right)^{\frac{1}{n}}
	.
\end{aligned}
\end{equation}
So we can see that $0$ is a $\mathcal{C}$-subsolution defined by Sz\'ekelyhidi~\cite{Szekelyhidi}.
In particular,
\begin{equation}
\label{inequality-2-9}
\sum^n_{i = 1} \frac{\partial\mathfrak{f}}{\partial \lambda_i} (\bm{\lambda})
\geq
n \left(\prod^n_{i = 1} \frac{\partial\mathfrak{f}}{\partial\lambda_i} (\bm{\lambda})\right)^{\frac{1}{n}}
\geq
n c^{\frac{1}{n}} .
\end{equation}

In this paper, we shall also adopt an Alexandrov-Bakelman-Pucci maximum principle in Sz\'ekelyhidi~\cite{Szekelyhidi}.
\begin{proposition}[Proposition 11 in \cite{Szekelyhidi}]
\label{proposition-abp}
Let $v : B (\bm{0},1) \to \mathbb{R}$ be smooth such that $v (\bm{0} ) + \epsilon \leq \inf_{\partial B (\bm{0},1)} v$, where $\epsilon > 0$ and $B (\bm{0},1) \subset \mathbb{R}^n$ denotes the unit ball.
Define the set
\begin{equation}
	P := \left\{\bm{x} \in B(\bm{0},1) \;\bigg|\; |D v (\bm{x})| < \frac{\epsilon}{2} , \; v (\bm{y}) \geq v (\bm{x}) + D v(\bm{x}) \cdot (\bm{y} - \bm{x}) \;\; \forall \bm{y} \in B (\bm{0} , 1) \right\}.
\end{equation}
Then there exists a dimensional constant $c_0 > 0$ such that
\begin{equation}
	c_0 \epsilon^n \leq \int_P \det D^2 v .
\end{equation}

\end{proposition}

There are $C^\infty$ estimates for Equation~\ref{equation-true} with smooth data (see for example \cite{Guan2014, Guan2023, Szekelyhidi}).
We only need to set up an continuity path to solve Equation~\ref{equation-true} with smooth data, which was actually implemented by \cite{Szekelyhidi} (Proposition 26). For completeness, we rephrase the continuity method argument. We shall consider that for $0 \leq t \leq 1$,
\begin{equation}
\label{continuity}
\mathfrak{F} (\chi + \nabla^2 \varphi_t) = e^{b_t} e^{t f + (1 - t) \tilde f}
\end{equation}
with smooth $f$ and $\tilde f = \ln\mathfrak{F} (\chi) $.
At the maximum point of $\varphi_t$,
\begin{equation}
e^{b_t} e^{t f + (1 - t) \tilde f}
\leq
e^{\tilde f}
,
\end{equation}
and hence
\begin{equation}
e^{b_t} \leq e^{t (\tilde f - f)} \leq e^{\max (\tilde f - f)^+} .
\end{equation}
By \eqref{inequality-2-14}, $0$ is a $\mathcal{C}$-subsolution to
\begin{equation}
\mathcal{F} (\chi + \nabla^2 \varphi) = e^{\max (\tilde f - f)^+} e^{\max \tilde f + \max f}.
\end{equation}
This implies that the $C^\infty$ estimates hold true for \eqref{continuity} independent of $t \in [0,1]$.
By a standard argument, there is a solution pair $(\varphi_t,b_t)$ to \eqref{continuity} for all $0 \leq t \leq 1$.

\medskip
\section{$L^\infty$ estimate}
\label{L-estimate}

In this section, we shall derive the $L^\infty$ estimate via Alexandrov-Bakelman-Pucci maximum principle.

\begin{theorem}
\label{theorem-3-1}
Suppose that $f$ is smooth and $\int_M e^{nf} (1 + n |f|)^p d vol < +\infty $ with $p > 0$.
Let $\varphi$ be a smooth solution to Equation~\eqref{equation-hessian}.
Then there is a constant $C$ depending on $\int_M e^{nf} (1 + n |f|)^p d vol $ and geometric data such that $- \varphi < C$.
In particular, there is a constant $C$ depending on $\Vert e^f \Vert_{L^{nq}}$ with $q > 1$ and geometric data such that $- \varphi < C$.
\end{theorem}

\begin{proof}

At any point $\bm{x} \in M$ with $\sigma g + \nabla^2 \varphi \geq 0$,
\begin{equation}
\label{inequality-3-1}
	e^f
	\geq
	\sum_{i,j} \mathfrak{F}^{ij} (\chi + \nabla^2 \varphi) \left(\sigma g_{ij} + \nabla^2_{ij} \varphi\right)
	\geq  n \left(c S_n (\bm{\lambda} (\sigma g + \nabla^2 \varphi))\right)^{\frac{1}{n}}
	.
\end{equation}
With a local chart around $\bm{x}$, we can rewrite \eqref{inequality-3-1} as
\begin{equation}
\label{inequality-3-2}
	\det \left(\sigma g + \nabla^2 \varphi \right) \leq  \frac{e^{nf}}{n^n c} \det g .
\end{equation}

Since $(M,g)$ is compact without boundary, there exists a constant $r_0 > 0$ such that for any point $\bm{x}_0 \in M$, there is a normal coordinate system $(x^1,\cdots,x^n)$ with $\bm{x}_0$ as the origin such that $g_{ij} (\bm{0}) = \delta_{ij}$ and
\begin{equation}
  0
  <
  \frac{1}{4} \nabla^2 |\bm{x}|^2
  \leq   g
  \leq   \nabla^2 |\bm{x}|^2
  ,
  \qquad
  \text{ in } B (\bm{0} , r_0)
  .
\end{equation}
Without loss of generality, we may assume that $r_0 = 1$.

Now we let $\bm{x}_0$ be the minimum point of $\varphi$ on $M$.
We define on $B (\bm{0},1)$
\begin{equation}
  u (\bm{x}) := \varphi (\bm{x}) - \varphi (\bm{0}) + \epsilon |\bm{x}|^2 ,
\end{equation}
and
\begin{equation}
  P
  :=
  \left\{ \bm{x} \in B (\bm{0} ,1) \; \bigg| \;   |D u (\bm{x})| < \frac{\epsilon}{2} ,\;  u (\bm{y}) \geq u (\bm{x}) + D u (\bm{x}) \cdot (\bm{y} - \bm{x})  \;\;\forall \bm{y} \in B (\bm{0} , 1)\right\}
  ,
\end{equation}
where $\epsilon > 0$ is to be specified later.
In $P$, we can see that
\begin{equation}
\label{inequality-3-6}
\begin{aligned}
	0
	\leq
	\left[ D^2_{ij} u\right]
	&=
	\left[ \nabla^2_{ij} u + \sum_k \Gamma^k_{ij} u_k \right]
	\\
	&\leq
	\left[ \nabla^2_{ij} \varphi + \epsilon \nabla^2_{ij} |\bm{x}|^2 + C (g) |\nabla \varphi| g_{ij}\right]
	\\
	&\leq
	\left[
	\sigma g_{ij} + \nabla^2_{ij} \varphi - \sigma g_{ij} + 4 \epsilon g_{ij } + \frac{C (g)}{2} \epsilon g_{ij}
	\right]
	\\
	&\leq
	\left[\sigma g_{ij} + \nabla^2_{ij} \varphi\right]
	,
\end{aligned}
\end{equation}
if we choose $\epsilon = \frac{\sigma}{8 + C (g)}$.
Without loss of generality, we may assume that $\epsilon < 1$.
Moreover,
\begin{equation}
	0
	\geq
	u (\bm{x})  - D u (\bm{x}) \cdot \bm{x}
	\geq
	u (\bm{x}) - \frac{\epsilon}{2} |\bm{x}|
	\geq
	u (\bm{x}) - \frac{\epsilon}{2}
	\geq
	\varphi (\bm{x}) - \varphi (\bm{0}) - \frac{\epsilon}{2} ,
\end{equation}
that is,
\begin{equation}
\label{inequality-3-8}
	\varphi (\bm{x})
	\leq \inf_M \varphi + \frac{\epsilon}{2}
	.
\end{equation}

It is easy to see that
\begin{equation}
	u (\bm{0}) + \epsilon
 	= \epsilon
 	\leq \inf_{\partial B (\bm{0},1)} u .
\end{equation}
By Proposition~\ref{proposition-abp} and \eqref{inequality-3-2}\eqref{inequality-3-6},
\begin{equation}
\label{inequality-3-10}
\begin{aligned}
c_0 \epsilon^n
&\leq
\int_P \det D^2 u
\\
&\leq
\int_P \det \left(\sigma g + \nabla^2 \varphi\right)
\\
&\leq
 \frac{1}{n^n c} \int_P   e^{nf} \det g
\\
&\leq
C' (c,g) \int_P e^{nf} d\, vol
.
\end{aligned}
\end{equation}
Since $\varphi \leq 0$,
\begin{equation}
\label{inequality-3-11}
\begin{aligned}
	c_0 \epsilon^n
	&\leq
	C' (c,g)  \int_P \frac{\ln^p \left(-\varphi + 1 \right)}{\ln^p (- \inf_M \varphi + 1 - \frac{\epsilon}{2})} e^{nf} d\, vol
	\\
	&=
	\frac{C' (c,g) }{\ln^p (- \inf_M \varphi + 1 - \frac{\epsilon}{2})} \int_P \ln^p (-\varphi + 1) e^{nf} d \, vol
	\\
	&=
 	\frac{C' (c,g)  2^p}{\ln^p (- \inf_M \varphi + 1 - \frac{\epsilon}{2})} \int_P \left( \frac{\ln (-\varphi + 1)}{2}\right)^p e^{nf} d \, vol
	.
\end{aligned}
\end{equation}
The generalized Young's inequality tells us that for any $p > 0$
\begin{equation}
	\left(\frac{\ln (- \varphi + 1)}{2}\right)^p e^{nf}
	\leq
	e^{nf} (1 + n |f|)^p + C (p) (-\varphi + 1) ,
\end{equation}
and hence from \eqref{inequality-3-11},
\begin{equation}
\begin{aligned}
	c_0 \epsilon^n
	&\leq
 	\frac{C' (c,g)  2^p}{\ln^p (- \inf_M \varphi + 1 - \frac{\epsilon}{2})}
 	\left(\int_P e^{nf} (1 + n |f|)^p d \, vol  + C (p) \int_P (-\varphi + 1) d\, vol\right)
 	\\
 	&\leq
	\frac{C' (c,g)  2^p}{\ln^p (- \inf_M \varphi + 1 - \frac{\epsilon}{2})}
 	\left(\int_M e^{nf} (1 + n |f|)^p d \, vol  + C (p) (\Vert \varphi \Vert_{L^1} + vol (M) ) \right)
 	.
\end{aligned}
\end{equation}

Also, we can utilize H\"older inequality and obtain
\begin{equation}
\begin{aligned}
c_0 \epsilon^n
&\leq
C' (c,g)  \int_P \left(\frac{- \varphi + 1}{- \inf_M \varphi + 1 - \frac{\epsilon}{2}}\right)^{\frac{q - 1}{q}} e^{nf} d\, vol
\\
&\leq
\frac{C' (c,g) }{\left( - \inf_M \varphi + 1 - \frac{\epsilon}{2} \right)^{\frac{q - 1}{q}} } \int_M \left( - \varphi + 1 \right)^{\frac{q - 1}{q}} e^{nf} d\, vol
\\
&\leq
\frac{C' (c,g) }{\left( - \inf_M \varphi + 1 - \frac{\epsilon}{2} \right)^{\frac{q - 1}{q}} } \left(\Vert\varphi\Vert_{L^1} + vol (M)\right)^{\frac{q - 1}{q }} \left\Vert e^{nf}\right\Vert_{L^q}
.
\end{aligned}
\end{equation}

\end{proof}

To apply the above $L^\infty$ estimate to  Equation~\ref{equation-true}, it is crucial to determine the upper bound of  $e^b$.

\begin{corollary}
\label{corollary-upper-bound}
Suppose  that $\int_M e^f d vol > 0$. Then there is an upper bound for the constant $b$.
\end{corollary}

\begin{proof}

By \eqref{inequality-2-3},
\begin{equation}
\label{inequality-3-17}
	   \Delta \varphi
	  \geq
	\frac{1}{\frac{\partial\mathfrak{f}}{\partial\lambda_1} (\bm{1})}\left( e^b e^f - \mathfrak{f} (\bm{1}) \right) - S_1 (\bm{\lambda} (\chi))  + n
	.
\end{equation}
Integrating \eqref{inequality-3-17},
\begin{equation}
	   \frac{\partial\mathfrak{f}}{\partial\lambda_1} (\bm{1})  \int_M S_1 (\bm{\lambda} (\chi)) d\,vol + \mathfrak{f} (\bm{1}) vol (M)  > e^b \int_M e^f d\, vol .
\end{equation}

\end{proof}

The lower bound of $e^b$ is useful in constructing a solution pair.
\begin{corollary}
\label{corollary-lower-bound}
Suppose that $\int_M e^{nf} dvol < + \infty$. Then there is a lower bound for the constant $b$.
\end{corollary}

\begin{proof}

At any point $\bm{x} \in M$ with $\sigma g + \nabla^2 \varphi \geq 0$,
\begin{equation}
 	e^b e^f
 	\geq  n c^{\frac{1}{n}}  \left(S_n (\bm{\lambda} (\sigma g + \nabla^2 \varphi))\right)^{\frac{1}{n}}
 	.
\end{equation}
According to Inequality~\eqref{inequality-3-10} in the proof of Theorem~\ref{theorem-3-1},
\begin{equation}
\begin{aligned}
c_0 \epsilon^n
&\leq
C' (c,g)  e^{nb} \int_P e^{nf} d\, vol
&\leq
C' (c,g)  e^{nb} \int_M e^{nf} d\, vol
.
\end{aligned}
\end{equation}

\end{proof}

\medskip
\section{Stability estimate}
\label{S-estimate}

In this section, we plan to study the stability estimates, that is, the estimates of the difference of two solutions $\varphi_1$ and $\varphi_2$.
Assume that there are $C^2$ functions $\varphi_1$ and $\varphi_2$ such that
\begin{equation}
\label{equation-4-1}
\mathfrak{F} (\chi + \nabla^2 \varphi_1) = e^f , \qquad \sup_M \varphi_1 = 0 ,
\end{equation}
and
\begin{equation}
\label{equation-4-2}
\mathfrak{F} (\chi + \nabla^2 \varphi_2) = e^{\tilde f} , \qquad \sup_M \varphi_2 = 0 .
\end{equation}
When $e^{nf} \in L^q$ for some $q > 1$, we have a bound for $\Vert \varphi_1 \Vert_{L^\infty}$ by the result in Section~\ref{L-estimate}.
By concavity of $\mathfrak{F}$ and \eqref{inequality-2-6},
\begin{equation}
\begin{aligned}
e^f
	&\geq
	(1 - r) \mathfrak{F} (\chi + \nabla^2 \varphi_2) + r \mathfrak{F} \left(\chi + \frac{1}{r} \nabla^2 \left(\varphi_1 - (1 - r) \varphi_2\right)\right)
	\\
	&\geq
	r \mathfrak{F} \left(\chi + \frac{1}{r} \nabla^2 \left(\varphi_1 - (1 - r) \varphi_2\right)\right)
	\\
	&\geq
	r \mathfrak{F} \left(\sigma g + \frac{1}{r} \nabla^2 \left(\varphi_1 - (1 - r) \varphi_2\right)\right)
	\\
	&\geq
	n r  c^{\frac{1}{n}}  S^{\frac{1}{n}}_n \left(\bm{\lambda} \left(\sigma g + \frac{1}{r} \nabla^2 \left(\varphi_1 - (1 - r) \varphi_2\right)\right) \right)
	,
\end{aligned}
\end{equation}
wherever $\sigma g + \frac{1}{r}\nabla^2 (\varphi_1 - (1 - r) \varphi_2) \geq 0$
for any $r \in \left(0,\frac{1}{2}\right)$.

Let $\bm{x}_0$ be the minimum point of $\frac{1}{r}\varphi_1 - \frac{1 - r}{r}\varphi_2$ on $M$.
We define
\begin{equation}
  u (\bm{x}) := \frac{1}{r}\varphi_1 (\bm{x}) - \frac{1 - r}{r}\varphi_2 (\bm{x}) - \frac{1}{r}\varphi_1 (\bm{0}) + \frac{1 - r}{r}\varphi_2 (\bm{0}) + \epsilon |\bm{x}|^2 ,
\end{equation}
and
\begin{equation}
  P
  :=
  \left\{ \bm{x} \in B (\bm{0} ,1) \; \bigg| \;   |D u (\bm{x})| < \frac{\epsilon}{2} ,\;  u (\bm{y}) \geq u (\bm{x}) + D u (\bm{x}) \cdot (\bm{y} - \bm{x})  \;\;\forall \bm{y} \in B (\bm{0} , 1)\right\}
  .
\end{equation}
Moreover, in $P$
\begin{equation}
	0
	\geq
	u (\bm{x}) - \frac{\epsilon}{2} |\bm{x}|
	\geq
	\frac{1}{r}\varphi_1 (\bm{x}) - \frac{1 - r}{r}\varphi_2 (\bm{x}) - \frac{1}{r}\varphi_1 (\bm{0}) + \frac{1 - r}{r}\varphi_2 (\bm{0}) - \frac{\epsilon}{2} ,
\end{equation}
that is,
\begin{equation}
\label{inequality-4-7}
	\frac{1}{r}\varphi_1 (\bm{x}) - \frac{1 - r}{r}\varphi_2 (\bm{x})
	\leq \inf_M \left(\frac{1}{r}\varphi_1  - \frac{1 - r}{r}\varphi_2  \right) + \frac{\epsilon}{2}
	.
\end{equation}
As in Section~\ref{L-estimate}, we obtain that in $P$
\begin{equation}
\label{inequality-4-8}
	0
	\leq
	D^2 u
	\leq
	\sigma g  +  \frac{1}{r} \nabla^2 \varphi_1 - \frac{1 - r}{r} \nabla^2 \varphi_2
	,
\end{equation}
and hence
\begin{equation}
\label{inequality-4-9}
\begin{aligned}
c_0 \epsilon^n
&\leq
\int_P \det D^2 u
\\
&\leq
\int_P \det \left(\sigma g + \frac{1}{r}\nabla^2 \varphi_1 - \frac{1 - r}{r} \nabla^2 \varphi_2\right)
\\
&\leq
 \frac{1}{c n^n r^n} \int_P   e^{nf} \det g
\\
&\leq
\frac{C' (c,g)}{r^n} \int_P e^{nf} d\, vol
.
\end{aligned}
\end{equation}
For fixed $p > 0$, we choose
\begin{equation}
\label{inequality-4-10}
r := \Vert (\varphi_2 - \varphi_1)^+\Vert^{\frac{p (q - 1)}{nq + p (q - 1)}}_{L^p} .
\end{equation}

If $0 < r < \frac{1}{2}$, we can apply Alexandrov-Bakelman-Pucci maximum principle to \eqref{inequality-4-9}.
There are two cases to discuss respectively.
The first case is that
\begin{equation}
- \inf_M \left(\varphi_1 - (1 - r) \varphi_2\right) \leq (1 + \Vert \varphi_1 \Vert_{L^\infty}) r ,
\end{equation}
and we have
\begin{equation}
\varphi_2 - \varphi_1
\leq
- \varphi_1 + \varphi_2 - r \varphi_2
\leq
(1 + \Vert \varphi_1 \Vert_{L^\infty}) r
.
\end{equation}
With \eqref{inequality-4-7}, the other case is then
\begin{equation}
\label{inequality-4-12}
\begin{aligned}
	\frac{1 - r}{r} (\varphi_1 - \varphi_2)
	&\leq
	\frac{1}{r}\varphi_1  - \frac{1 - r}{r}\varphi_2   +  \Vert \varphi_1 \Vert_{L^\infty}
	\\
	&\leq \inf_M \left(\frac{1}{r}\varphi_1  - \frac{1 - r}{r}\varphi_2  \right) +  \Vert \varphi_1 \Vert_{L^\infty} + \frac{\epsilon}{2} 	
	< - \frac{1}{2} ,
\end{aligned}
\end{equation}
for any $\bm{x} \in P$.
Combining \eqref{inequality-4-9} and \eqref{inequality-4-12},
\begin{equation}
\label{inequality-4-13}
\begin{aligned}
	c_0 \epsilon^n
	&\leq
	\frac{C' (c,g)}{r^n} \int_P e^{nf} d\, vol
	\\
	&\leq
	\frac{C' (c,g)}{r^n} \int_P \left(\frac{- \frac{1}{r} \varphi_1 + \frac{1 - r}{r} \varphi_2 -  \Vert \varphi_1 \Vert_{L^\infty}}{- \inf_M \left(\frac{1}{r}\varphi_1  - \frac{1 - r}{r}\varphi_2  \right)  -  \Vert \varphi_1 \Vert_{L^\infty} - \frac{\epsilon}{2} 	} \right)^{\frac{p (q - 1)}{q}} e^{nf} d\, vol
	\\
	&=
	\frac{C' (c,g)}{r^n} \int_P \left(\frac{-  (1 - r) \varphi_1 - r \varphi_1 + (1 - r) \varphi_2 -  \Vert \varphi_1 \Vert_{L^\infty} r}{- \inf_M \left( \varphi_1  - (1 - r) \varphi_2  \right)  -  \Vert \varphi_1 \Vert_{L^\infty} r - \frac{\epsilon}{2} r	} \right)^{\frac{p (q - 1)}{q}} e^{nf} d\, vol
	\\
	&\leq
	\frac{C' (c,g)}{r^n} \int_P \left( \frac{-  (1 - r) \varphi_1 + (1 - r) \varphi_2 }{- \inf_M \left( \varphi_1  - (1 - r) \varphi_2  \right)  -  \Vert \varphi_1 \Vert_{L^\infty} r - \frac{\epsilon}{2} r } \right)^{\frac{p (q - 1)}{q}}	e^{nf} d\, vol
	\\
	&\leq
	\frac{C' (c,g)}{r^n} \int_P \left( \frac{\varphi_2 -    \varphi_1  }{- \inf_M \left( \varphi_1  - (1 - r) \varphi_2  \right)  -  \Vert \varphi_1 \Vert_{L^\infty} r - \frac{\epsilon}{2} r } \right)^{\frac{p (q - 1)}{q}}	e^{nf} d\, vol
	,
\end{aligned}
\end{equation}
for $p > 0$.
Applying H\"older inequality to \eqref{inequality-4-13}
\begin{equation}
\label{inequality-4-15}
\begin{aligned}
&\quad
	 c_0 \epsilon^n \left(- \inf_M \left( \varphi_1  - (1 - r) \varphi_2  \right)  - \Vert \varphi_1 \Vert_{L^\infty} r - \frac{\epsilon}{2} r\right)^{\frac{p (q - 1)}{q}}
	\\
	&\leq
	\frac{C' (c,g)}{ r^n}  \int_P (\varphi_2 - \varphi_1)^{\frac{p (q - 1)}{q}}  e^{nf} d\, vol
	\\	
	&\leq
	\frac{C' (c,g)}{ r^n}
	\left( \int_P (\varphi_2 - \varphi_1)^p d\, vol \right)^{\frac{q - 1}{q}} \Vert e^{nf} \Vert_{L^q}
	\\
	&\leq
	\frac{C' (c,g)}{ r^n}
	\Vert (\varphi_2 - \varphi_1)^+\Vert^{\frac{p (q - 1)}{q}}_{L^p}  \Vert e^{nf} \Vert_{L^q}
	.
\end{aligned}
\end{equation}
Substituting \eqref{inequality-4-10} into \eqref{inequality-4-15},
\begin{equation}
\label{inequality-4-16}
\begin{aligned}
&\quad
	c_0 \epsilon^n \left(- \inf_M \left( \varphi_1  - (1 - r) \varphi_2  \right) -   \Vert \varphi_1 \Vert_{L^\infty} r  - \frac{\epsilon}{2} r\right)^{\frac{p (q - 1)}{q}}
	\leq
	C' (c,g) r^{  \frac{p (q - 1)}{q}}   \Vert e^{nf} \Vert_{L^q}
	.
\end{aligned}
\end{equation}
Rewriting \eqref{inequality-4-16},
\begin{equation}
\label{inequality-4-17}
\begin{aligned}
	\varphi_2 - \varphi_1
	\leq
	(1 - r) \varphi_2 - \varphi_1
	\leq
	\left(\frac{C' (c,g) \Vert e^{nf} \Vert_{L^q}}{c_0 \epsilon^n}\right)^{\frac{q}{p (q - 1)}} r
	+
	\Vert \varphi_1 \Vert_{L^\infty} r
	+
	\frac{\epsilon}{2} r
	.
\end{aligned}
\end{equation}

If $r = 0$, then $\varphi_2 - \varphi_1 \leq 0$  everywhere on $M$.

If $r \geq \frac{1}{2}$, then we have
\begin{equation}
\varphi_2 - \varphi_1
\leq
- \varphi_1
\leq
\Vert \varphi_1 \Vert_{L^\infty}
\leq
2 \Vert \varphi_1 \Vert_{L^\infty} r .
\end{equation}

Therefore, we complete the proof of Theorem~\ref{theorem-stability-1}

\medskip
\section{Weak solutions}
\label{weak-solution}

In this section, we shall study the weak solution pair to Equation~\eqref{equation-true}.

\medskip

\subsection{Construction of a limit function pair}
\label{limit-function}

We shall first construct a limit function pair $(\varphi, b)$ with respect to Equation~\eqref{equation-true}.
We need a compact embedding property which was also proved for complex case~\cite{Sun202305}. For completeness, we state the proof here.
\begin{theorem}
\label{theorem-5-1}
The set of uniformly $L^\infty$-bounded $C^2$ solutions to Equation~\eqref{equation-true} is precompact in  $L^{q'}$ norm for any $1 \leq q' < + \infty$.
\end{theorem}
\begin{proof}

Since $\bm{\lambda} (\chi + \nabla^2 \varphi), \bm{\lambda} (\chi) \in \Gamma \subset \Gamma^1$ ,
\begin{equation}
\label{inequality-5-1}
\begin{aligned}
	\Vert \varphi \Vert_{L^\infty} \int_M  S_1 (\bm{\lambda} (\chi ))  d \, vol
	\geq
	\int_M - \varphi S_1 (\bm{\lambda} (\chi + \nabla^2 \varphi)) d\, vol
	\geq
	\Vert \nabla \varphi\Vert^2_{L^2}
	.
\end{aligned}
\end{equation}
If $\{\varphi_\alpha\}$ is a bounded set in $L^\infty$ norm, it is precompact in $L^1$ norm by Sobolev embedding.
So there exists a $L^1$ convergent sequence $\{\varphi_i\}$, and
\begin{equation}
\int_M |\varphi_i - \varphi_j |^{q'} d\, vol
\leq
\Vert \varphi_i - \varphi_j \Vert^{q' - 1}_{L^\infty} \int_M |\varphi_i - \varphi_j| d\, vol
\leq
C  \int_M |\varphi_i - \varphi_j| d\, vol
\to
0
,
\end{equation}
as $i,j $ approach $\infty$ for any $1 \leq q' < + \infty$.

\end{proof}

With the estimates in Section~\ref{L-estimate}, Section~\ref{S-estimate} and Theorem~\ref{theorem-5-1}, we can construct a continuous limit function pair with respect to Equation~\eqref{equation-true}.
For nonnegative function $e^f \in L^{qn} (M)$ with $q > 1$, there is a sequence  $\{e^{f_i}\}$ of smooth positive functions such that
\begin{equation}
\Vert e^{f_i} - e^f\Vert_{L^{nq}} < \frac{1}{2^i} .
\end{equation}
Then we can estimate
\begin{equation}
	\Vert e^{n f_i} \Vert_{L^q}
	=
	\Vert e^{f_i} \Vert^n_{L^{nq}}
	\leq
	\left( \Vert e^{f_i} - e^f\Vert_{L^{nq}} + \Vert e^f \Vert_{L^{nq}}\right)^n
	<
	\left( \frac{1}{2^i} + \Vert e^f \Vert_{L^{nq}}\right)^n
	,
\end{equation}
\begin{equation}
\begin{aligned}
	\int_M e^{nf_i} d \,vol
	\leq
	\left( \Vert e^{f_i} - e^f\Vert_{L^n} + \Vert e^f \Vert_{L^n} \right)^n
	<
	\left(\frac{1}{2^{i}} (vol (M))^{\frac{q - 1}{n q}} + \Vert e^f \Vert_{L^n}\right)^n
	,
\end{aligned}
\end{equation}
and
\begin{equation}
\int_M e^{f_i} d\, vol
=
\int_M e^f d\, vol + \int_M \left(e^{f_i} - e^f\right) d\, vol
\geq
\int_M e^f d\, vol - \frac{1}{2^i} \, (vol (M))^{\frac{n q - 1}{n q}} .
\end{equation}
When $\int_M e^{f} d vol > 0$, we have
\begin{equation}
\label{inequality-5-7}
\Vert e^{nf_i} \Vert_{L^q} \leq 2 \Vert e^{n f} \Vert_{L^q} ,
\end{equation}
\begin{equation}
\label{inequality-5-8}
\int_M e^{nf_i} d\, vol \leq 2 \int_M e^{nf} d\, vol ,
\end{equation}
and
\begin{equation}
\label{inequality-5-9}
\int_M e^{f_i} d\, vol \geq \frac{1}{2} \int_M  e^f d\, vol ,
\end{equation}
for $i$ sufficiently large.

We can find out a pair of smooth function $\varphi_i$ and real constant $b_i$ such that
\begin{equation}
	\mathfrak{F} (\chi + \nabla^2 \varphi_i) = e^{b_i} e^{f_i}, \qquad \sup \varphi_i = 0.
\end{equation}
Combining Corollary~\ref{corollary-lower-bound}, Corollary~\ref{corollary-upper-bound},  \eqref{inequality-5-8} and \eqref{inequality-5-9}, we obtain
\begin{equation}
\label{b-5-11}
	\left(\frac{c_0 \epsilon^n}{2 C' (c,g) \int_M e^{nf} d\, vol}\right)^{\frac{1}{n}}
	\leq
	e^{b_i}
	<
	\frac{2 \frac{\partial\mathfrak{f}}{\partial\lambda_1} (\bm{1}) \int_M S_1 (\bm{\lambda } (\chi)) d\, vol + 2 \mathfrak{f} (\bm{1}) vol (M)}{\int_M e^f d\, vol}
	,
\end{equation}
for sufficiently large $i$.
By passing to a subsequence, $e^{b_i}$ can be assumed to be convergent to a constant $e^b$ between
$$
\left(\frac{c_0 \epsilon^n}{ 2 C' (c,g) \int_M e^{nf} d\, vol}\right)^{\frac{1}{n}}
\text{ and }
\frac{ 2 \frac{\partial\mathfrak{f}}{\partial\lambda_1} (\bm{1}) \int_M S_1 (\bm{\lambda } (\chi)) d\, vol +  2 \mathfrak{f} (\bm{1}) vol (M)}{\int_M e^f d\, vol}
.
$$
Then we can do the estimation
\begin{equation}
\begin{aligned}
	\left\Vert e^{b_i} e^{f_i} - e^b e^f \right\Vert_{L^{nq}}
	&\leq
	e^{b_i} \Vert e^{f_i} - e^f\Vert_{L^{nq}} + \left| e^{b_i} - e^b \right| \Vert e^f \Vert_{L^{nq}}
	\\
	&<
	\frac{ \frac{\partial\mathfrak{f}}{\partial\lambda_1} (\bm{1}) \int_M S_1 (\bm{\lambda } (\chi)) d\, vol +  \mathfrak{f} (\bm{1}) vol (M)}{2^{i - 1} \int_M e^f d\, vol}
	+
	\left| e^{b_i} - e^b \right| \Vert e^f \Vert_{L^{nq}},
\end{aligned}
\end{equation}
for sufficiently large $i$. Consequently, $e^{b_i} e^{f_i}$ is $L^{nq}$-convergent to $e^b e^f$.
As a consequence of Theorem~\ref{theorem-3-1}, $\{\varphi_i\}$ is uniformly bounded.

By Theorem~\ref{theorem-5-1}, $\{\varphi_i\}$ is convergent in $L^1$ norm by passing to a subsequence if necessary.
Then we can choose an increasing sequence $\{j_k\}$ such that $\forall i \geq j_k$,
\begin{equation}
	\Vert \varphi_i - \varphi_{j_k}\Vert_{L^1} \leq \frac{1}{2^{k\frac{nq + q - 1}{q - 1}}} .
\end{equation}
By the stability estimate Theorem~\ref{theorem-stability-1}, there is a constant $C > 0$ such that
\begin{equation}
	\Vert \varphi_i - \varphi_j\Vert_{L^\infty} \leq C \Vert \varphi_i - \varphi_j\Vert^{\frac{q - 1}{n q + q - 1}}_{L^1} .
\end{equation}
Therefore,
\begin{equation}
\label{inequality-5-15}
	\Vert \varphi_{j_l} - \varphi_{j_k}\Vert_{L^\infty} \leq \frac{C}{2^k} , \qquad\forall\, l \geq k,
\end{equation}
which implies that $\{\varphi_{j_k}\}$ is uniformly convergent to a continuous function $\varphi$. Moreover, $\varphi$ is also a $W^{1,2}$ function by Alaoglu's Theorem.

\medskip
\subsection{The uniqueness of constant $b$}
We shall show that the constant $e^b$ is unique for a fixed $e^f$.
\begin{theorem}

Assume that we have two smooth approximation sequences $\{(\varphi_i, b_i, f_i)\}$ and $\{(\tilde \varphi_i , \tilde b_i, \tilde f_i)\}$ satisfying
\begin{equation}
\label{equation-5-16}
\mathfrak{F} (\chi + \nabla^2 \varphi_i) = e^{b_i} e^{f_i} , \qquad \sup \varphi_i = 0 ,
\end{equation}
and
\begin{equation}
\label{equation-5-17}
\mathfrak{F} (\chi + \nabla^2 \tilde \varphi_i) = e^{\tilde b_i} e^{\tilde f_i} , \qquad \sup \tilde \varphi_i = 0 ,
\end{equation}
such that as $i \to \infty$
\begin{equation}
\Vert e^{f_i} -  e^f\Vert_{L^n} \to 0, \qquad \Vert e^{\tilde f_i} -  e^f\Vert_{L^n} \to 0
\end{equation}
and
\begin{equation}
b_i \to b, \qquad \tilde b_i \to \tilde  b ,
\end{equation}
for some constants $b, \tilde b \in \mathbb{R}$. Then $b = \tilde b$.
\end{theorem}

\begin{proof}

From Equality~\eqref{equation-5-16} and \eqref{equation-5-17}, we derive that for $0 < r < \frac{1}{2}$,
\begin{equation}
\begin{aligned}
e^{b_i} e^{f_i}
	&\geq
	(1 - r) \mathfrak{F} (\chi + \nabla^2 \tilde \varphi_i) + r \mathfrak{F} \left(\chi + \frac{1}{r} \nabla^2 \left(\varphi_i - (1 - r) \tilde \varphi_i\right)\right)
	\\
	&\geq
	(1 - r) e^{\tilde b_i} e^{\tilde f_i}
	+
	r  n c^{\frac{1}{n}}  S^{\frac{1}{n}}_n \left(\bm{\lambda} \left(\sigma g + \frac{1}{r} \nabla^2 \left(\varphi_i - (1 - r) \tilde \varphi_i \right)\right) \right)
	,
\end{aligned}
\end{equation}
wherever $\sigma g + \frac{1}{r} \nabla^2 \left(\varphi_i - (1 - r) \tilde \varphi_i \right) \geq 0$.
Then we have
\begin{equation}
\frac{1 }{r} \left(e^{b_i} e^{f_i} -  (1 - r) e^{\tilde b_i} e^{\tilde f_i}\right)
\geq
 n c^{\frac{1}{n}}   S^{\frac{1}{n}}_n \left(\bm{\lambda} \left(\sigma g + \frac{1}{r} \nabla^2 \left(\varphi_i - (1 - r) \tilde \varphi_i \right)\right) \right)
,
\end{equation}
for $0 < r < \frac{1}{2}$.
As previous,  let $\bm{x}_0$ be the minimum point of $\frac{1}{r} \left(\varphi_i - (1 - r) \tilde \varphi_i \right)$ on $M$,
and hence define
\begin{equation*}
  u (\bm{x})
  :=
  \frac{1}{r} \left(\varphi_i (\bm{x}) - (1 - r) \tilde \varphi_i (\bm{x})\right)
  -
  \frac{1}{r} \left(\varphi_i (\bm{0}) - (1 - r) \tilde \varphi_i (\bm{0})\right)
  +
  \epsilon |\bm{x}|^2 ,
\end{equation*}
and
\begin{equation*}
  P
  :=
  \left\{ \bm{x} \in B (\bm{0} ,1) \; \bigg| \;   |D \varphi (\bm{x})| < \frac{\epsilon}{2} ,\;  u (\bm{y}) \geq u (\bm{x}) + D \varphi (\bm{x}) \cdot (\bm{y} - \bm{x})  \;\;\forall \bm{y} \in B (\bm{0} , 1)\right\}
  .
\end{equation*}
Similar to \eqref{inequality-4-9}, we obtain that
\begin{equation}
\label{inequality-5-20}
\begin{aligned}
c_0 \epsilon^n
&\leq
C' (c, g) \int_P \left(  \frac{1 }{r} \left(e^{b_i}e^{f_i} - (1 - r) e^{\tilde b_i} e^{\tilde f_i} \right) \right)^n d vol
\\
&\leq
C' (c, g) \int_M \left(  \frac{1 }{r} \left(e^{b_i} e^{f_i} - (1 - r) e^{\tilde b_i} e^{\tilde f_i} \right)^+ \right)^n d vol
.
\end{aligned}
\end{equation}

If $b < \tilde b$,
we may choose $r = 1 - e^{b - \tilde b} > 0$ and let $i \to \infty$ in \eqref{inequality-5-20},
\begin{equation}
\begin{aligned}
c_0 \epsilon^n
&\leq
C' (c, g) \int_M \left(\frac{1}{r} \left(e^b e^{f} - (1 - r) e^{\tilde b} e^{ f}\right)^+\right)^n d vol
= 0
,
\end{aligned}
\end{equation}
which is a contradiction.
Therefore, it has to be that $b = \tilde b$.

\end{proof}


\medskip
\subsection{Weak $C^0$ solution}


Adapting the definition from Chen~\cite{Chen2000} (see also Yuan~\cite{Yuan2022}), we can say that the limit function pair is a weak $C^0$ solution.
\begin{definition}
\label{definition-5-2}
A continuous function $\varphi \in C(M)$ is a weak $C^0$ solution to  Equation~\eqref{equation-hessian} with respect to $L^{nq}$ norm,
if for any $\xi > 0$, there exists a $C^2$ function $\varphi_\xi$ and a continuous function ${f_\xi}$ such that
\begin{equation}
|\varphi - \varphi_\xi| < \xi
\end{equation}
and
\begin{equation}
\mathfrak{F} (\chi + \nabla^2 \varphi_\xi) = e^{f_\xi} > 0 ,
\end{equation}
where $\left\Vert e^{f_\xi} - e^f\right\Vert_{L^{nq}} \to 0$ as $\xi \to 0+$ with $q > 1$.

\end{definition}

From \eqref{b-5-11}, we can see that the term on the right hand side of Equation~\eqref{equation-true} satisfies that
\begin{equation}
\begin{aligned}
&\quad
	\left(\frac{c_0 \epsilon^n}{2 C' (c,g) }\right)^{\frac{1}{n}} \frac{e^f}{\left(vol (M)\right)^{\frac{q - 1}{q}}\Vert e^f \Vert_{L^{nq}}}
	\leq
	e^{b}  e^f
	\\
	&
	\qquad
	<
	2	\left( \frac{\partial\mathfrak{f}}{\partial\lambda_1} (\bm{1}) \Vert S_1 (\bm{\lambda } (\chi)) \Vert + \mathfrak{f} (\bm{1}) vol (M)\right) \frac{e^f}{\Vert e^f \Vert_{L^1}} ,
\end{aligned}
\end{equation}
if $e^f \in L^{nq} (M)$ and $\int_M e^f d\, vol > 0$.
Then we have
\begin{equation}
\begin{aligned}
\left(\frac{c_0 \epsilon^n}{2 C' (c,g) }\right)^{\frac{1}{n}} \frac{1}{(vol (M))^{\frac{q - 1}{q}}}
\frac{\Vert e^f \Vert_{L^1}}{\Vert e^f \Vert_{L^{nq}}}
&\leq
\Vert e^b e^f\Vert_{L^1} \\
&\leq 2	\left( \frac{\partial\mathfrak{f}}{\partial\lambda_1} (\bm{1}) \Vert S_1 (\bm{\lambda } (\chi)) \Vert + \mathfrak{f} (\bm{1}) vol (M)\right) ,
\end{aligned}
\end{equation}
and
\begin{equation}
\begin{aligned}
\left(\frac{c_0 \epsilon^n}{2 C' (c,g) }\right)^{\frac{1}{n}} \frac{1}{(vol (M))^{\frac{q - 1}{q}}}
&\leq
\Vert e^b e^f\Vert_{L^{nq}}
\\
&\leq
2	\left( \frac{\partial\mathfrak{f}}{\partial\lambda_1} (\bm{1}) \Vert S_1 (\bm{\lambda } (\chi)) \Vert + \mathfrak{f} (\bm{1}) vol (M)\right) \frac{\Vert e^f \Vert_{L^{nq}}}{\Vert e^f \Vert_{L^1}}
.
\end{aligned}
\end{equation}
We can adapt the argument for complex Monge-Amp\`ere equation~\cite{Kolodziej2005} to show that weak $C^0$ solutions are equicontinuous.
\begin{theorem}
For constant $K > 0$, let
$$
R_K := \left\{e^{f} \in L^{nq} \;\bigg|\; \frac{\Vert e^f\Vert_{L^{nq}}}{\Vert e^f\Vert_{L^1} } < K\right\} .
$$
Then the solution set to Equation~\eqref{equation-true} with  the term $e^f\in R_K$ on the right hand side are equicontinuous on $M$.
\end{theorem}
\begin{proof}

According to Definition~\ref{definition-5-2}, we can see that the weak $C^0$ solution is a limit of a uniformly convergent sequence of  smooth functions. Without loss of generality, we may assume that the terms on the right hand side are positive and smooth.

Assume that for some $\delta > 0$, there is a sequence $(\varphi_i , \bm{x}_i,\bm{y}_i)$ so that
\begin{equation*}
d (\bm{x}_i , \bm{y}_i) < \frac{1}{i} ,
\end{equation*}
and
\begin{equation*}
\varphi_i (\bm{x}_i) - \varphi_i (\bm{y}_i) > \delta .
\end{equation*}
When $e^f \in R_K$, the solutions are uniformly bounded in $L^\infty$. By Theorem~\ref{theorem-5-1} and the compactness of $M$, it is possible to choose a subsequence $(\varphi_{i_j} , \bm{x}_{i_j}, \bm{y}_{i_j})$ such that there is a point $\bm{z} \in M$ and a function  $\varphi$ satisfying that
$\varphi_{i_j} \to \varphi $ in $L^1 $,
$\bm{x}_{i_j} \to \bm{z} $ and $\bm{y}_{i_j} \to \bm{z} $,
as $j \to +\infty$.
By Theorem~\ref{theorem-stability-1}, $\varphi_{i_j}$ converges uniformly, and thus $\varphi$ is continuous. 
Therefore, 
\begin{equation*}
\begin{aligned}
	\left|\varphi_{i_j} (\bm{x}_{i_j}) - \varphi_{i_j} (\bm{y}_{i_j})\right|
	&\leq \left| \varphi_{i_j} (\bm{x}_{i_j}) - \varphi (\bm{x}_{i_j}) \right| + \left| \varphi (\bm{x}_{i_j}) - \varphi(\bm{z})  \right| \\
	&\quad + \left|  \varphi (\bm{z}) - \varphi (\bm{y}_{i_j})  \right| + \left| \varphi (\bm{y}_{i_j}) - \varphi_{i_j} (\bm{y}_{i_j})\right|
	\\
	&\leq 2 \Vert \varphi - \varphi_{i_j}\Vert_{L^\infty}  + \left| \varphi (\bm{x}_{i_j}) - \varphi(\bm{z})  \right| + \left|  \varphi (\bm{z}) - \varphi (\bm{y}_{i_j})  \right| \to 0 ,
\end{aligned}
\end{equation*}
which is a contradiction.

\end{proof}

\medskip
\subsection{Viscosity solution}

The viscosity solution can be defined as follows.
\begin{definition}
Let $\varphi \in C(M)$ and $e^{f} \in C(M)$.

Function $\varphi$ is a viscosity subsolution to Equation~\eqref{equation-hessian} if for any point $\bm{x} \in M$ and any $C^2$ function $u$ such that $u - \varphi$ has a local minimum at $\bm{x}$, it holds true that at $\bm{x}$
\begin{equation}
	\mathfrak{f} \left(\bm{\lambda} (\chi + \nabla^2 u)\right) \geq e^{f} .
\end{equation}

Function $\varphi$ is a viscosity supersolution to Equation~\eqref{equation-hessian} if for any point $\bm{x} \in M$ and any $C^2$ function $u$ such that $u - \varphi$ has a local maximum at $\bm{x}$, one of the following holds true that at $\bm{x}$
\begin{equation}
	\bm{\lambda} (\chi + \nabla^2 u) \not\in \Gamma ,
\end{equation}
or
\begin{equation}
	\bm{\lambda} (\chi + \nabla^2 u) \in \Gamma \qquad \text{ and } \qquad \mathfrak{f} \left(\bm{\lambda} (\chi + \nabla^2 u)\right) \leq e^{f} .
\end{equation}

Function $\varphi$ is a viscosity solution to Equation~\eqref{equation-hessian} if it is both a viscosity subsolution and a viscosity supersolution.

\end{definition}

When $e^f \in C(M)$ is not identically $0$, we can construct a positive smooth approximation sequence $\{e^{f_i}\}$ of $e^f$ in $C^0$ norm.
As in Subsection~\ref{limit-function}, we can construct a limit function pair $(\varphi,b)$.
According to the standard argument in \cite{Lions1983}, $(\varphi,b)$ is a viscosity solution pair to Equation~\ref{equation-true}.
The proof of Theorem~\ref{main-theorem} is complete.

\medskip
\section{Lipschitz estimate}
\label{Interior gradient estimate}

In this section, we shall investigate the Lipschitz estimate when the term $e^f$ on the right hand side is Lipschitz.

\medskip
\subsection{Interior gradient estimate on Riemannian manifolds}

We shall study the interior gradient estimate for a more general form of equations,
\begin{equation}
\label{equation-6-1}
\mathfrak{F} \left(\chi + \nabla^2 \varphi\right) = \mathfrak{f} \left(\bm{\lambda} (\chi + \nabla^2 \varphi)\right) = \psi \left( \bm{x} , \varphi (\bm{x} ) , d\varphi(\bm{x} ) \right) ,
\end{equation}
for a $C^1$ function $\psi$ on $M \times \mathbb{R} \times T^* M$.
We assume here that $\mathfrak{f}$ satisfies the following conditions:
\begin{enumerate}[(i)]

\item $\mathfrak{f} (\bm{\lambda})$ is a function in $ C^1 (\Gamma) $,
where $\Gamma \subset \Gamma^1 := \{\bm{\lambda} \in \mathbb{R}^n \,| \, \lambda_1 + \cdots + \lambda_n > 0  \}$ is an open symmetric convex domain
;

\item $\mathfrak{f} (\bm{\lambda})$ is invariant under permutations of the components of $\bm{\lambda}$;

\item $\frac{\partial \mathfrak{f}}{\partial \lambda_i} \geq 0$, for all $1\leq i\leq n$ and $\bm{\lambda} \in \Gamma$;

\item $\sum^n_{i = 1} \frac{\partial \mathfrak{f}}{\partial \lambda_i} (\bm{\lambda}) \lambda_i \geq 0$, for all $\bm{\lambda} \in \Gamma$;

\item $\lim_{t \to +\infty} \mathfrak{f} (t\bm{1}) > \sup_{B_r (\bm{0})} \psi$;

\item $\mathfrak{f} (\bm{\lambda})$ is concave.
\end{enumerate}
These conditions are naturally satisfied by many equations.
It is easy to see that Condition (i)-(iii), (v) and (vi) can be derived from the settings of Equation~\ref{equation-hessian}.

%
%
%
%
%
%
%
%
%
%
%
%
%
%
From Condition (v), there are constants $L > 0$, $\epsilon > 0$ depending on $\sup_{B_r(0)} \psi $ such that
\begin{equation*}
	\mathfrak{f} (L, \cdots , L) > \sup_{B_r (0)} \psi + \epsilon .
\end{equation*}
When $\bm{\lambda} (\chi + \nabla^2 \varphi) \in \Gamma$, Condition (vi) implies that
\begin{equation*}
	- \sum_{i,j} \frac{\partial \mathfrak{F} }{\partial \varphi_{ij}} \left(\chi_{ij} + \nabla^2_{ij} \varphi\right)
	\geq  - L \sum_{i,j} \frac{\partial \mathfrak{F}}{\partial \varphi_{ij}} g_{ij} + \mathfrak{F} (Lg) - \mathfrak{F} (\chi + \nabla^2 \varphi)
	> - L \sum_{i,j} \frac{\partial \mathfrak{F}}{\partial \varphi_{ij}} g_{ij} + \epsilon,
\end{equation*}
and hence by Condition (iv),
\begin{equation}
\label{V-2}
	\sum_{i,j} \frac{\partial \mathfrak{F}}{\partial \varphi_{ij}} g_{ij} > \frac{\epsilon}{L} .
\end{equation}
%
%

%
%
When $\bm{\lambda} (\chi + \nabla^2 \varphi) \in \partial \Gamma$, we further impose an extra assumption that
there is a strictly increasing  function $h\in C^1 (\mathbb{R})$
and a  function $\mathfrak{\tilde f} \in C^1(\bar\Gamma\setminus ( \bar\Gamma^n \cap \partial \Gamma))$
such that
\(
	\mathfrak{\tilde f} (\bm{\lambda}) = h \circ \mathfrak{f} (\bm{\lambda})
\)
in $\Gamma$
and
\begin{equation*}
\label{g-3}
	\sum_i \frac{\partial \mathfrak{\tilde f}}{\partial \lambda_i} > 0, \qquad \text{when } \bm{\lambda} \in \partial\Gamma\setminus \bar\Gamma^n.
\end{equation*}
Equation~\eqref{equation-6-1} is equivalent to
\begin{equation}
\label{equivalent-equation}
	\mathfrak{\tilde F} (\chi + \nabla^2 \varphi) = h \circ \psi (x, \varphi , d\varphi) ,
\end{equation}
where
\begin{equation*}
	\mathfrak{\tilde F} (X) =  \limsup_{\substack{ \bm{\lambda} \in \Gamma \\ \bm{\lambda} \rightarrow \bm{\lambda} (X) } }  \mathfrak{\tilde f} (\bm{\lambda} (X)) .
\end{equation*}
%
%
%
%
%
%
%
For more details and discussions on the structure conditions, we refer the readers to \cite{SuiSun20}.
However, we do not need the condition that
$\frac{\partial \mathfrak{f}}{\partial \lambda_j} > \delta \sum_i \frac{\partial \mathfrak{f}}{\partial \lambda_i}$ when $\lambda_j < 0$, for some $\delta > 0 $.

In local coordinates $(x^1, \cdots, x^n)$, denote $\omega = \sum_i p_i d x^i \in T^*_{\bm{x}} M$. We can express $\psi$ as $\psi (x^1 , \cdots , x^n, t , p_1, \cdots , p_n)$.
For convenience, we may express
\begin{equation*}
\begin{aligned}
\partial_i &:= \frac{\partial}{\partial x^i} , \\
\varphi_i &:= \nabla_{\partial_i} \varphi = \frac{\partial \varphi}{\partial x^i} ,\\
\varphi_{ij} &:= \nabla^2_{\partial_i\partial_j}  \varphi = \Big<\nabla^2 \varphi , \frac{\partial}{\partial x^i}\otimes \frac{\partial}{\partial x^j}\Big> , \\
\varphi_{ijk} &:= \nabla^3_{\partial_i\partial_j\partial_k} \varphi = \Big<\nabla^3 \varphi , \frac{\partial}{\partial x^i} \otimes \frac{\partial}{\partial x^j}  \otimes \frac{\partial}{\partial x^k}\Big> ,\\
& \cdots \cdots
\end{aligned}
\end{equation*}
We define one form
$$
d^H \psi (\bm{x}, t , \omega) := \sum_i d^H_{x^i} \psi (\bm{x}, t, \omega) d x^i
$$
and $(1,0)$ tensor
$$
\frac{\partial \psi}{\partial \omega} := \sum_j \frac{\partial \psi}{\partial p_j} \frac{\partial}{\partial x^j} ,
$$
where
\begin{equation*}
	d^H_{x^i} \psi (\bm{x}, t, \omega)
	:= \frac{\partial \psi}{\partial x^i} + \sum_{j} \frac{\partial \psi}{\partial p_j} \Big< \nabla_{\frac{\partial}{\partial x^i}} \frac{\partial}{\partial x^j}, \omega\Big>
	= \frac{\partial \psi}{\partial x^i} + \sum_{j,k} \frac{\partial \psi}{\partial p_j} \Gamma^k_{ij} p_k.
\end{equation*}
Therefore we have
\begin{equation*}
\begin{aligned}
	\frac{\partial }{\partial x^i} \psi (\bm{x},\varphi , d \varphi) &= \frac{\partial \psi}{\partial t} \varphi_i +  \frac{\partial \psi}{\partial x^i} + \sum_j \frac{\partial \psi}{\partial p_j} \frac{\partial^2 \varphi}{\partial x^j \partial x^i} \\
	&= \frac{\partial \psi}{\partial t} \varphi_i +  d^H_{x^i} \psi + \sum_j \frac{\partial \psi}{\partial p_j} \varphi_{ji} .
\end{aligned}
\end{equation*}
We can verify that $d^H \psi (\bm{x}, t, \omega)$ does not depend on the choice of coordinates.
If we induce another coordinate chart $\bm{y} = (y^1 , \cdots ,y^n)$, there is a new coordinate chart $(\bm{y},\bm{q})$ for $T^* M$.  The transition map is
\begin{equation*}
	(y^1, \cdots , y^n , q_1 , \cdots , q_n) = \left(y^1 (\bm{x}) , \cdots , y^n (\bm{x}), \sum_j p_j \frac{\partial x^j}{\partial y^1} , \cdots ,  \sum_j p_j \frac{\partial x^j}{\partial y^n}\right) ,
\end{equation*}
and hence $\frac{\partial q^m}{\partial p_j} = \frac{\partial x^j}{\partial y^m}$.
Then we can derive that
\begin{equation*}
\begin{aligned}
	d^H_{x^i}\psi
	&=
	\frac{\partial}{\partial x^i} \left( \psi (\bm{x} , t , \bm{p}) \right)
	+
	\sum_{j} \frac{\partial \psi}{\partial p_j} \Big< \nabla_{\frac{\partial}{\partial x^i}} \frac{\partial}{\partial x^j} , \omega \Big>
	\\
	&=
	\frac{\partial}{\partial x^i} \left( \psi (\bm{y} , t , \bm{q}) \right)
	+
	\sum_{j,k,l,m} \frac{\partial \psi}{\partial q_m} \frac{\partial q_m}{\partial p_j} \frac{\partial y^k}{\partial x^i} \Big< \nabla_{\frac{\partial}{\partial y^k}} \frac{\partial y^l}{\partial x^j} \frac{\partial }{\partial y^l}, \omega \Big>
	\\
	&=
	\sum_j \frac{\partial y^j}{\partial x^i} \frac{\partial \psi}{\partial y^j} +  \sum_{j,k} \frac{\partial \psi}{\partial q_j} \frac{\partial}{\partial x^i} \left(\frac{\partial x^k}{\partial y_j}\right) p_k
	+
	\sum_{j,k,l,m} \frac{\partial \psi}{\partial q_m} \frac{\partial x^j}{\partial y^m}  \frac{\partial y^k}{\partial x^i}  \frac{\partial}{\partial y^k} \left(\frac{\partial y^l}{\partial x^j}\right) q_l
	\\
	&\qquad
	+
	\sum_{j,k,l,m} \frac{\partial \psi}{\partial q_m} \frac{\partial x^j}{\partial y^m} \frac{\partial y^k}{\partial x^i}  \frac{\partial y^l}{\partial x^j} \Big< \nabla_{\frac{\partial}{\partial y^k}} \frac{\partial }{\partial y^l}, \omega \Big>
	\\
	&=
	\sum_j \frac{\partial y^j}{\partial x^i} \left(\frac{\partial \psi}{\partial y^j} +  \sum_l \frac{\partial \psi}{\partial q_l}  \Big< \nabla_{\frac{\partial}{\partial y^j}} \frac{\partial }{\partial y^l}, \omega \Big>\right)
	\\
	&=
	\sum_j \frac{\partial y^j}{\partial x^i} d^H_{y^j} \psi
	.
\end{aligned}
\end{equation*}

We can derive the following interior gradient estimates for Equation~\eqref{equation-6-1} by adapting the argument of Chou and Wang~\cite{CW1}.
\begin{theorem}
\label{Theorem-6-1}
Let $(M, g)$ be a Riemannian manifold of dimension $n$ and $\chi$ a smooth symmetric $(0,2)$ tensor on $M$.
Also suppose that $B_r (\bm{x}_0) \subset M$ and $\varphi \in C^3 (B_r (\bm{x_0}))$
for $ r < \min\{ inj(\bm{x}_0) , d(\bm{x}_0, \partial M) \} $, where 
$inj(\bm{x}_0)$ is the injectivity radius at point $\bm{x}_0$. Let $\varphi$ be a  solution to Equation~\eqref{equation-6-1} satisfying
\begin{equation}
\label{condition-6}
	\left|d^H \psi (\bm{x}, \varphi , \omega) \right| + \left|\frac{\partial \psi}{\partial t} (\bm{x}, \varphi, \omega)\right| |\omega| + \left|\frac{\partial \psi}{\partial \omega} (\bm{x}, \varphi,\omega)\right| |\omega|^2 \leq h (|\omega|^3), \quad \text{as } |\omega| \rightarrow +\infty ,
\end{equation}
where $h(t) = o(t)$ as $t\rightarrow + \infty$.
Then we have
\begin{equation*}
	|\nabla \varphi (\bm{x}_0)| \leq C ,
\end{equation*}
where $C$ depends on $\epsilon$, $L$, $r$, $\sup_{B_r (\bm{x}_0)} | \varphi |$ and geometric data.
\end{theorem}

\begin{proof}


For simplicity, we may use $X$ and $\mathfrak{F}^{ij}$ to express $\chi + \nabla^2 \varphi$ and  $\frac{\partial \mathfrak{F}}{\partial X_{ij}}$ respectively.

Now we start to prove the interior gradient estimate by adapting the idea of Chou and Wang~\cite{CW1}. Let $\varphi \in C^3(B_r (\bm{0}))$ be a solution to Equation~\eqref{equation-6-1} and $B_r (\bm{0}) \subset M$. Let us consider the function
\begin{equation*}
\label{gradient-test-function}
	G(\bm{x}) := \frac{1}{2} \ln |\nabla \varphi|^2 + \tau (\varphi) + \ln \rho (\bm{x}) ,
\end{equation*}
where
\begin{equation*}
\tau(\varphi) = - \frac{1}{3} \ln (2K - \varphi), \qquad \text{for } K = \sup_{B_r(\bm{0})} |\varphi| ,
\end{equation*}
and
\begin{equation*}
\rho(\bm{x}) = \left(1 - \frac{d^2(\bm{x},\bm{0})}{r^2}\right)^+, \qquad \text{for } r < \min\{ inj(\bm{0}) , d(\bm{0}, \partial M) \} .
\end{equation*}
It is easy to see that
\begin{equation*}
\begin{aligned}
\tau' &= \frac{1}{3(2K - \varphi)} , \\
\tau'' &= \frac{1}{3 (2K - \varphi)^2} ,
\end{aligned}
\end{equation*}
and hence
\begin{equation*}
\tau'' - 2{\tau'}^2 = \frac{1}{9 (2K - \varphi)^2} = {\tau'}^2.
\end{equation*}
Suppose that $G$ attains its maximum at some point $\bm{p} \in B_r (\bm{0})$. 
For convenience, we choose a local chart near $\bm{p}$ such that $g_{ij} = \delta_{ij}$ and $X_{i j}$ is diagonal at $\bm{p}$. Without loss of generality, we may assume that $X_{11} \geq X_{22} \geq \cdots \geq X_{nn}$. Obviously, $\left[\mathfrak{F}^{ij}\right]$ is also diagonal. Therefore, at the point $\bm{p}$, we have
\begin{equation}
\label{gradient-derivative-1}
\frac{1}{2} \frac{\partial_i (|\nabla \varphi|^2)}{|\nabla \varphi|^2}  + \tau' \varphi_i + \frac{\rho_i}{\rho} = 0 ,
\end{equation}
and
\begin{equation}
\label{gradient-derivative-2}
\frac{\nabla^2_{\partial_i\partial_i} |\nabla \varphi|^2}{2 |\nabla \varphi|^2} - \frac{\partial_i |\nabla \varphi|^2 \partial_i |\nabla \varphi|^2}{2 |\nabla \varphi|^4} + \tau'' \varphi_i \varphi_i + \tau' \varphi_{ii} + \frac{\rho_{ii}}{\rho} - \frac{\rho_i \rho_i}{\rho^2} \leq 0 .
\end{equation}
Through an orthogonal transformation, we can also find a new local chart $\{y^1,y^2,\cdots,y^n\}$ such that
$
g\left(\partial_{y^i},\partial_{y^j}\right) = \delta_{ij}
$
and	
$
\varphi_{y^1} = |\nabla \varphi|
$. By \eqref{gradient-derivative-1},
\begin{equation*}
	\varphi_{y^1 y^i} = - |\nabla \varphi|\left(\tau' \varphi_{y^i} - \frac{2 d d_{y^i}}{r^2\rho}\right) .
\end{equation*}
In particular,
\begin{equation*}
	\varphi_{y^1 y^1} =  - \tau' |\nabla \varphi|^2 + \frac{2 d d_{y^1}}{r^2\rho} |\nabla \varphi| .
\end{equation*}
Assuming that at $\bm{p}$,
$$
|\nabla \varphi| \rho >   \sqrt{ 18 \sup_{B_r(\bm{0})} \sup_{|\bm{\xi}| = 1} \chi (\bm{\xi} , \bm{\xi}) K }  + \frac{3 6 K }{r}
$$
and hence
\begin{equation}
\label{term-1}
X_{nn}  \leq \chi_{y^1 y^1}  - \frac{1}{3 (2K - \varphi)} |\nabla \varphi|^2 + \frac{2 |\nabla \varphi|}{r \rho}
   \leq -  \frac{1}{18 K} |\nabla \varphi|^2 .
\end{equation}
Therefore, $\bm{\lambda} (X) \in \bar\Gamma \setminus \bar \Gamma^n$ at $\bm{p}$.

By direct calculation,
\begin{equation*}
	\nabla^2_{\partial_i\partial_i} \big(|\nabla \varphi|^2\big) = 2 \sum_{k} \varphi_{iik} \varphi_k + 2 \sum_{k,l} R_{ikil}  \varphi_k \varphi_l + 2 \sum_{k}  \varphi^2_{ki} .
\end{equation*}
Combining \eqref{gradient-derivative-1} and \eqref{gradient-derivative-2},
\begin{equation}
\label{gradient-derivative-3}
\begin{aligned}
	0 &\geq
	\frac{\nabla^2_{\partial_i\partial_i} |\nabla \varphi|^2}{2 |\nabla \varphi|^2} - \frac{\partial_i |\nabla \varphi|^2 \partial_i |\nabla \varphi|^2}{2 |\nabla \varphi|^4} + \tau'' \varphi^2_i + \tau' \varphi_{ii} + \frac{\rho_{ii}}{\rho} - \frac{\rho^2_i }{\rho^2} \\
	&\geq
	\frac{1}{ |\nabla \varphi|^2} \left(	 \sum_{k} \varphi_{iik} \varphi_k +  \sum_{k,l} R_{ikil}  \varphi_k \varphi_l +  \sum_{k}  \varphi^2_{ki} \right) + \tau' \varphi_{ii} + \frac{\rho_{ii}}{\rho} - \frac{7 \rho^2_i }{\rho^2} 	
	.
\end{aligned}
\end{equation}

There are two cases in consideration: (1) $\bm{\lambda} (\chi + \nabla^2 \varphi) \in \Gamma$ and (2) $\bm{\lambda} (\chi + \nabla^2 \varphi) \in \partial\Gamma$.

{\em Case  1:} $\bm{\lambda} (\chi + \nabla^2 \varphi) \in \Gamma$ .
Differentiating Equation~\eqref{equation-6-1},
\begin{equation*}
	\sum_{i,j} \mathfrak{F}^{ij} X_{ijk} = \frac{\partial \psi}{\partial t} \varphi_k +  d^H_{x^k} \psi + \sum_j \frac{\partial \psi}{\partial p_j} \varphi_{jk} .
\end{equation*}
Then at $\bm{p}$,
\begin{equation*}
\begin{aligned}
	\sum_{i} \mathfrak{F}^{ii} \nabla^2_{\partial_i\partial_i} \left(|\nabla \varphi|^2\right)
	&= 2   \frac{\partial \psi}{\partial t} |\nabla \varphi|^2 + 2 \sum_k   d^H_{x^k} \psi  \varphi_k +  \sum_{j} \frac{\partial \psi}{\partial p_j} \partial_j (|\nabla \varphi|^2) \\
	&\quad - 2 \sum_{i,k} \mathfrak{F}^{ii}\chi_{iik} \varphi_k + 2 \sum_{i,k,l} \mathfrak{F}^{ii} R_{ikil} \varphi_k \varphi_l + 2 \sum_{i,k} \mathfrak{F}^{ii} \varphi^2_{ki} 	
	.
\end{aligned}
\end{equation*}

Multiplying \eqref{gradient-derivative-3} by $\mathfrak{F}^{ii}$ and summing these terms over index $i$,
\begin{equation}
\label{inequality-1}
\begin{aligned}
	0
	&\geq
		  \frac{\partial \psi}{\partial t}  + \frac{1}{|\nabla \varphi|^2}  \sum_k   d^H_{x^k} \psi  \varphi_k + \frac{1}{2 |\nabla \varphi|^2}  \sum_{j} \frac{\partial \psi}{\partial p_j} \partial_j (|\nabla \varphi|^2) \\
			&\quad - \frac{1}{|\nabla \varphi|^2}  \sum_{i,k} \mathfrak{F}^{ii}\chi_{iik} \varphi_k + \frac{1}{|\nabla \varphi|^2}  \sum_{i,k,l} \mathfrak{F}^{ii} R_{ikil} \varphi_k \varphi_l + \frac{1}{|\nabla \varphi|^2} \sum_{i,k} \mathfrak{F}^{ii} \varphi^2_{ki} 	\\
			&\quad + \tau' \sum_i \mathfrak{F}^{ii} \varphi_{ii} + \sum_i \mathfrak{F}^{ii} \frac{\rho_{ii}}{\rho} - 7 \sum_i \mathfrak{F}^{ii} \frac{ \rho^2_i }{\rho^2}  \\
		&\geq
		  \frac{\partial \psi}{\partial t}  + \frac{1}{|\nabla \varphi|^2}  \sum_k   d^H_{x^k} \psi  \varphi_k - \tau' \sum_{j} \frac{\partial \psi}{\partial p_j} \varphi_j + 2 \sum_{j} \frac{\partial \psi}{\partial p_j} \frac{ d d_j}{ r^2\rho} - \frac{1}{|\nabla \varphi|^2}  \sum_{i,k} \mathfrak{F}^{ii}\chi_{iik} \varphi_k  \\
			&\quad + \frac{1}{|\nabla \varphi|^2}  \sum_{i,k,l} \mathfrak{F}^{ii} R_{ikil} \varphi_k \varphi_l + \frac{1}{2 |\nabla \varphi|^2} \sum_{i} \mathfrak{F}^{ii} X^2_{ii}  - \frac{1}{|\nabla \varphi|^2} \sum_{i,k} \mathfrak{F}^{ii}  \chi^2_{ki} 	 \\
			&\quad + \tau' \sum_i \mathfrak{F}^{ii} \varphi_{ii} + \sum_i \mathfrak{F}^{ii} \frac{\rho_{ii}}{\rho} - 7 \sum_i \mathfrak{F}^{ii} \frac{ \rho^2_i }{\rho^2} .
\end{aligned}
\end{equation}
Putting \eqref{inequality-1} and \eqref{term-1} together,
\begin{equation}
\label{inequality-2}
\begin{aligned}
	0
	&\geq
		  \frac{\partial \psi}{\partial t}  - \frac{|d^H \psi|}{|\nabla \varphi|}   - \tau' |\nabla \varphi| \left|\frac{\partial \psi}{\partial \omega}\right| -  \frac{  2}{ r\rho} \left| \frac{\partial \psi}{\partial \omega} \right| - \frac{1}{|\nabla \varphi|^2}  \sum_{i,k} \mathfrak{F}^{ii}\chi_{iik} \varphi_k  \\
			&\quad + \frac{1}{2 n |\nabla \varphi|^2} X^2_{nn}\sum_{i} \mathfrak{F}^{ii}  + \frac{1}{|\nabla \varphi|^2}  \sum_{i,k,l} \mathfrak{F}^{ii} R_{ikil} \varphi_k \varphi_l  - \frac{1}{|\nabla \varphi|^2} \sum_{i,k} \mathfrak{F}^{ii}  \chi^2_{ki} 	 \\
			&\quad + \tau' \sum_i \mathfrak{F}^{ii} X_{ii} -  \tau' \sum_i \mathfrak{F}^{ii} \chi_{ii}  + \sum_i \mathfrak{F}^{ii} \frac{\rho_{ii}}{\rho} - 7 \sum_i \mathfrak{F}^{ii} \frac{ \rho^2_i }{\rho^2} \\
	&\geq
		\frac{|\nabla \varphi|^2}{648 n K^2} \sum_{i} \mathfrak{F}^{ii}
		 - \left| \frac{\partial \psi}{\partial t} \right| \frac{L}{\epsilon} \sum_i \mathfrak{F}^{ii} - \frac{|d^H \psi|}{|\nabla \varphi|}   \frac{L}{\epsilon} \sum_i \mathfrak{F}^{ii} -  \frac{|\nabla \varphi|}{3K} \left|\frac{\partial \psi}{\partial \omega}\right|  \frac{L}{\epsilon} \sum_i \mathfrak{F}^{ii} \\
			&\quad -  \frac{  2}{ r\rho} \left| \frac{\partial \psi}{\partial \omega} \right|  \frac{L}{\epsilon} \sum_i \mathfrak{F}^{ii}  - \frac{|\nabla \chi|}{|\nabla \varphi|}  \sum_{i} \mathfrak{F}^{ii}  +  \inf_{\substack{|\bm{\xi}| = 1\\ |\bm{\eta}| = 1}} R_{\bm{\xi\eta\xi\eta}}\sum_{i} \mathfrak{F}^{ii}  - \frac{|\chi|^2}{|\nabla \varphi|^2} \sum_{i} \mathfrak{F}^{ii} 	 \\
			&\quad -  \frac{1}{3K}  \sup_{|\bm{\xi}| = 1} \chi (\bm{\xi} , \bm{\xi}) \sum_i \mathfrak{F}^{ii}  -  \frac{1}{r^2 \rho^2} \left(28 +  \sup_{|\bm{\xi}| = 1} \nabla^2 d^2 (\bm{\xi} , \bm{\xi})\right)\sum_i \mathfrak{F}^{ii}  .
\end{aligned}
\end{equation}
Dividing \eqref{inequality-2} by $\frac{1}{\rho^2 K^2}\sum_i \mathfrak{F}^{ii}$,
\begin{equation*}
\label{inequality-3}
\begin{aligned}
	0
	&\geq
		\frac{|\nabla \varphi|^2 \rho^2}{648 n }
		 - \frac{\left| \frac{\partial \psi}{\partial t} \right|}{|\nabla \varphi|^2} \frac{L K^2}{\epsilon} \rho^2 |\nabla \varphi|^2 - \frac{|d^H \psi|}{|\nabla \varphi|^3}   \frac{L K^2}{\epsilon}  \rho^2 |\nabla \varphi|^2  -    \frac{\left|\frac{\partial \psi}{\partial \omega}\right| }{|\nabla \varphi|} \frac{L K}{3 \epsilon} \rho^2 |\nabla \varphi|^2 \\
			&\quad -   \frac{\left| \frac{\partial \psi}{\partial \omega} \right| }{|\nabla \varphi|} \frac{2 L K^2}{r \epsilon}   \rho |\nabla \varphi| - \frac{|\nabla \chi|}{\rho |\nabla \varphi|} \rho^3 K^2 +   \Bigg\{\inf_{\substack{|\bm{\xi}| = 1\\ |\bm{\eta}| = 1}} R_{\bm{\xi\eta\xi\eta}} \Bigg\}^- K^2 - \frac{|\chi|^2}{\rho^2 |\nabla \varphi|^2}    K^2	 \\
			&\quad -  \frac{  K}{3}  \sup_{|\bm{\xi}| = 1} \chi (\bm{\xi} , \bm{\xi})    -  \frac{K^2}{r^2 } \left(28 +  \sup_{|\bm{\xi}| = 1} \nabla^2 d^2 (\bm{\xi} , \bm{\xi})\right)  .
\end{aligned}
\end{equation*}
So we can obtain the interior gradient estimate under assumption~\eqref{condition-6}.

%

{\em Case 2:} $\bm{\lambda} (\chi + \nabla^2 \varphi) \in \partial \Gamma$.
We consider Equation~\eqref{equivalent-equation} instead.

From \eqref{V-2},
\begin{equation*}
\begin{aligned}
	\sum_{i,j} \mathfrak{\tilde F}^{ij} (\chi + \nabla^2 \varphi + tg) g_{ij} > h' \circ \mathfrak{F} (\chi + \nabla^2 \varphi + tg) \frac{\epsilon}{L}
\end{aligned}
\end{equation*}
for $t> 0$. Letting $t \rightarrow 0+$,
\begin{equation*}
	\sum_{i,j} \mathfrak{\tilde F}^{ij} g_{ij} \geq h' (\psi) \frac{\epsilon}{L} .
\end{equation*}

Multiplying \eqref{gradient-derivative-3} by $\mathfrak{\tilde F}^{ii}$ and summing these terms over index $i$,
\begin{equation*}
\begin{aligned}
	0
		&\geq
			\frac{|\nabla \varphi|^2}{648 n K^2} \sum_{i} \mathfrak{\tilde F}^{ii}
			 - h' (\psi) \left| \frac{\partial \psi}{\partial t} \right|  - h'(\psi) \frac{|d^H \psi|}{|\nabla \varphi|}    -  h'(\psi) \frac{|\nabla \varphi|}{3K} \left|\frac{\partial \psi}{\partial \omega}\right|  -  \frac{  2 h'(\psi) }{ r\rho} \left| \frac{\partial \psi}{\partial \omega} \right|    \\
				&\quad  - \frac{|\nabla \chi|}{|\nabla \varphi|}  \sum_{i} \mathfrak{\tilde F}^{ii}  +  \inf_{\substack{|\bm{\xi}| = 1\\ |\bm{\eta}| = 1}} R_{\bm{\xi\eta\xi\eta}}\sum_{i} \mathfrak{\tilde F}^{ii}  - \frac{|\chi|^2}{|\nabla \varphi|^2} \sum_{i} \mathfrak{\tilde F}^{ii} -  \frac{1}{3K}  \sup_{|\bm{\xi}| = 1} \chi (\bm{\xi} , \bm{\xi}) \sum_i \mathfrak{\tilde F}^{ii}  	 \\
				&\quad -  \frac{1}{r^2 \rho^2} \left(28 +  \sup_{|\bm{\xi}| = 1} \nabla^2 d^2 (\bm{\xi} , \bm{\xi})\right)\sum_i  \mathfrak{\tilde F}^{ii}  .
\end{aligned}
\end{equation*}
If $h'(\psi) = 0$,
\begin{equation*}
\begin{aligned}
	0
		&\geq
			\frac{|\nabla \varphi|^2\rho^2 }{648 n }  - \frac{|\nabla \chi|\rho^2}{|\nabla \varphi|}    +  \rho^2 K^2 \Bigg\{ \inf_{\substack{|\bm{\xi}| = 1\\ |\bm{\eta}| = 1}} R_{\bm{\xi\eta\xi\eta}} \Bigg\}^- - \frac{\rho^2 |\chi|^2 K^2}{|\nabla \varphi|^2}  \\
			&\quad -  \frac{\rho^2 K }{3}  \sup_{|\bm{\xi}| = 1} \chi (\bm{\xi} , \bm{\xi}) 	 -  \frac{K^2}{r^2 } \left(28 +  \sup_{|\bm{\xi}| = 1} \nabla^2 d^2 (\bm{\xi} , \bm{\xi})\right) ;
\end{aligned}
\end{equation*}
if $h' (\psi) > 0$, we can again obtain \eqref{inequality-3}.

\end{proof}

\medskip
\subsection{Interior gradient estimate on $\mathbb{R}^n$}

A particular case of Equation~\eqref{equation-6-1} is
\begin{equation}
\label{equation-R}
\mathfrak{F} (D^2 \varphi) = \mathfrak{f} (\bm{\lambda} (D^2 \varphi)) = \psi (x , \varphi , D \varphi) .
\end{equation}
Caffarelli, Nirenberg and Spruck~\cite{CNSIII} established existence results on the corresponding Dirichlet problem in $\mathbb{R}^n$.
Chou and Wang~\cite{CW1} established the interior gradient estimates for $k$-Hessian equations (see also \cite{Wang2009}), and Chen~\cite{Chen2015} proved similar results for Hessian quotient equations.
\begin{theorem}
\label{theorem-R}
Suppose that $\varphi \in C^3 (B_r (\bm{0}))$ is a solution to Equation~\eqref{equation-R} such that
\begin{equation*}
\label{condition-7}
	|D_{\bm{x}} \psi| + |\psi_\varphi|~|\bm{p}| + |D_{\bm{p}} \psi|~|\bm{p}|^2 \leq  h(|\bm{p}|^3) \quad \text{as } |\bm{p}| \rightarrow +\infty,
\end{equation*}
where $h (t) = o(t)$ as $t \rightarrow +\infty$. Then we have
\begin{equation*}
	|D \varphi (\bm{0})| \leq C ,
\end{equation*}
where $C$ depends on $\epsilon$, $L$, $r$ and $\sup_{B_r(\bm{0})} |\varphi (\bm{x})|$.

In particular, if $\psi$ is constant,
\begin{equation*}
	|D \varphi (\bm{0})| \leq C \frac{\sup_{B_r(\bm{0})} |\varphi|}{r} .
\end{equation*}
\end{theorem}
\begin{remark}
 With Theorem~\ref{theorem-R}, a Liouville type result can be achieved as in \cite{CW1}.
\end{remark}

\begin{proof}

We shall consider the function
\begin{equation*}
\label{R-test-function}
	G(\bm{x}) := \frac{1}{2} \ln |D \varphi|^2 + \tau (\varphi) + \ln \rho (\bm{x}) ,
\end{equation*}
where
\begin{equation*}
\tau (\varphi) = - \frac{1}{3} \ln (2K - \varphi), \qquad \text{for } K = \sup_{B_r(\bm{0})} |\varphi| ,
\end{equation*}
and
\begin{equation*}
\rho(x) = \left(1 - \frac{|\bm{x}|^2}{r^2}\right)^+.
\end{equation*}
Suppose that $G$ attains its maximal value at some point $\bm{p} \in B_r (\bm{0})$. For convenience, we choose a set of coordinates $\bm{x} = (x^1, \cdots, x^n)$ so that $D^2 \varphi$ is diagonal at $\bm{p}$ and $\varphi_{11} \geq \cdots \geq \varphi_{nn}$. Therefore, at point $\bm{p}$,
\begin{equation}
\label{equality-R-1}
\frac{1}{2} \frac{(|D \varphi|^2)_i}{|D \varphi|^2} + \tau' \varphi_i + \frac{\rho_i}{\rho} = 0 ,
\end{equation}
and
\begin{equation}
\label{inequality-R-1}
\frac{(|D \varphi|^2)_{ii}}{2 |D \varphi|^2} - \frac{(|D \varphi|^2)_i(|D \varphi|^2)_i}{2 |D \varphi|^4} + \tau'' \varphi_i \varphi_i + \tau' \varphi_{ii} + \frac{\rho_{ii}}{\rho} - \frac{\rho_i \rho_i}{\rho^2} \leq 0.
\end{equation}
From \eqref{equality-R-1}, in another set of coordinates $\bm{y} = (y^1,\cdots,y^n)$, $\varphi_{y^1} = |D \varphi|$ and
\begin{equation*}
	\varphi_{y^1y^i} = - |D \varphi| \left(\tau' \varphi_{y^i} - \frac{2 y_i}{r^2 \rho}\right) .
\end{equation*}
We assume that at point $\bm{p}$, $|D \varphi| \rho > \frac{36 K}{r}$. Then,
\begin{equation}
\label{inequality-R-5}
\begin{aligned}
	\varphi_{nn} \leq - \tau' |D \varphi|^2 + \frac{2 y_1 |D \varphi|}{r^2 \rho} \leq - \frac{|D \varphi|^2}{18 K} .
\end{aligned}
\end{equation}
There are two cases to deal with.

{\em Case 1:} $\bm{\lambda} (D^2 \varphi) \in \Gamma$.
From \eqref{inequality-R-1},
\begin{equation}
\label{inequality-R-2}
\begin{aligned}
	0 &\geq\frac{1}{2 |D \varphi|^2}  \sum_i \mathfrak{F}^{ii} (|D \varphi|^2)_{ii} -\frac{1}{2 |D \varphi|^4} \sum_i \mathfrak{F}^{ii} (|D \varphi|^2)_i(|D \varphi|^2)_i \\
	&\quad + \tau'' \sum_i \mathfrak{F}^{ii} \varphi_i \varphi_i + \tau' \sum_i \mathfrak{F}^{ii} \varphi_{ii} + \frac{1}{\rho} \sum_i \mathfrak{F}^{ii} \rho_{ii} - \frac{1}{\rho^2}  \sum_i \mathfrak{F}^{ii} \rho_i \rho_i \\
	&\geq- \frac{|D_{\bm{x}}  \psi|}{|D \varphi|}  + \psi_\varphi - \frac{1}{3K} |D \varphi| |D_{\bm{p}} \psi| - \frac{2}{r\rho}|D_{\bm{p} } \psi|  + \frac{\varphi^2_{nn}}{n |D \varphi|^2}  \sum_i \mathfrak{F}^{ii}   - \frac{30}{r^2 \rho^2}  \sum_i \mathfrak{F}^{ii}    .
\end{aligned}
\end{equation}
Substituting \eqref{inequality-R-5} into \eqref{inequality-R-2},
\begin{equation*}
\label{inequality-R-6}
\begin{aligned}
	0 &\geq - \frac{|D_{\bm{x}}  \psi|}{|D \varphi|}   + \psi_\varphi - \frac{1}{3K} |D \varphi| |D_{\bm{p}} \psi| - \frac{2}{r\rho}|D_{\bm{p} } \psi|  + \frac{|D \varphi|^2}{324 n K^2}  \sum_i \mathfrak{F}^{ii}   - \frac{30}{r^2 \rho^2}  \sum_i \mathfrak{F}^{ii}  .
\end{aligned}
\end{equation*}
By Inequality~\eqref{V-2},
\begin{equation}
\label{inequality-R-8}
\begin{aligned}
	0 &\geq - \frac{|D_{\bm{x}}  \psi|}{|D \varphi|} \frac{L}{\epsilon} \sum_i \mathfrak{F}^{ii}  - | \psi_\varphi |\frac{L}{\epsilon} \sum_i \mathfrak{F}^{ii}  - \frac{1}{3K} |D \varphi| |D_{\bm{p}} \psi| \frac{L}{\epsilon} \sum_i \mathfrak{\mathfrak{F}}^{ii}  - \frac{2}{r\rho}|D_{\bm{p} } \psi| \frac{L}{\epsilon} \sum_i \mathfrak{\mathfrak{F}}^{ii}  \\
	&\quad + \frac{|D \varphi|^2}{324 n K^2}  \sum_i \mathfrak{F}^{ii}   - \frac{30}{r^2 \rho^2}  \sum_i \mathfrak{F}^{ii} .
\end{aligned}
\end{equation}
Dividing \eqref{inequality-R-8} by $\frac{\sum_i \mathfrak{F}^{ii}}{\rho^2 K^2}$,
\begin{equation}
\label{inequality-5}
\begin{aligned}
	\frac{| \psi_\varphi | }{|D \varphi|^2} \frac{  L K^2 }{\epsilon} |D \varphi|^2 \rho^2  +\frac{ |D_{\bm{p}} \psi| }{|D \varphi|} \frac{K L}{3 \epsilon} |D \varphi|^2 \rho^2 + \frac{ |D_{\bm{p} } \psi|}{|D \varphi|} \frac{2 L K^2}{r \epsilon} |D \varphi| \rho \\
	 + \frac{|D_{\bm{x}}  \psi|}{|D \varphi|^3} \frac{L K^2}{\epsilon} |D \varphi|^2 \rho^2 + \frac{30 K^2}{r^2 }  \geq \frac{|D \varphi|^2 \rho^2}{324 n}     .
\end{aligned}
\end{equation}
So we are able to obtain the interior gradient estimate under \eqref{condition-7}.

{\em Case 2:} $\bm{\lambda} (D^2 \varphi) \in \partial \Gamma$.
We consider an equivalent equation
\begin{equation*}
	\mathfrak{\tilde F} (D^2 \varphi) = h \circ \psi (x , \varphi , D \varphi) .
\end{equation*}
Multiplying \eqref{gradient-derivative-3} by $\mathfrak{\tilde F}^{ii}$ and summing these terms over index $i$,
\begin{equation*}
\label{inequality-R-9}
	0
	\geq
	- \frac{h' |D_{\bm{x}}  \psi|}{|D \varphi|}  + h' \psi_\varphi - \frac{h'}{3K} |D \varphi| |D_{\bm{p}} \psi| - \frac{2 h'}{r\rho}|D_{\bm{p} } \psi|  + \frac{\varphi^2_{nn}}{n |D \varphi|^2}  \sum_i \mathfrak{\tilde F}^{ii}   - \frac{30}{r^2 \rho^2}  \sum_i \mathfrak{\tilde F}^{ii}    .
\end{equation*}
If $h'(\psi) = 0$,
\begin{equation*}
	\frac{30}{r^2}
	\geq
	 \frac{|D \varphi|^2 \rho^2 }{324 n K^2}        ;
\end{equation*}
if $h' (\psi) > 0$, we can again obtain \eqref{inequality-5}.

{\em A particular case:} $\psi$ is constant.
From \eqref{inequality-R-2} and \eqref{inequality-R-9}, we have
\begin{equation*}
	\frac{30 K^2}{r^2} \geq \frac{|D \varphi|^2\rho^2}{324 n} ,
\end{equation*}
and hence
\begin{equation*}
\label{inequality-R-7}
	|D \varphi|\rho \leq 99 \sqrt{n}\frac{K}{r} .
\end{equation*}

\end{proof}

\medskip
\subsection{Lipschitz estimate}

If $e^f$ is locally Lipschitz in a open set $U \subset M$, then we can find out a  sequence $\{e^{f_i}\}$ of positive smooth functions such that $e^{f_i} \to e^f$ both in $L^\infty$  and  $C^{0,1}_{loc} (U)$.

To apply the interior gradient estimate, we just need to ensure that Condition~(iv) holds true. If we further assume that \eqref{condition-5-32} holds,
then by concavity
\begin{equation}
(t - 1) \sum_i \frac{\partial \mathfrak{f}}{\partial\lambda_i} (\bm{\lambda}) \lambda_i
\geq
\mathfrak{f} (t \bm{ \lambda}) - \mathfrak{f} (\bm{\lambda})
>
- \mathfrak{f} (\bm{\lambda}) .
\end{equation}
Letting $t \to +\infty$,
\begin{equation}
\label{inequality-5-34}
\sum_i \frac{\partial \mathfrak{f}}{\partial\lambda_i} (\bm{\lambda}) \lambda_i \geq 0 .
\end{equation}
%
%
%
So we can obtain a local Lipschitz bound for some weak solutions.
Therefore, Theorem \ref{Lipschitz solution} is proved.

Indeed, we can obtain a uniform Lipschitz estimate for approximation equations when $e^f$ is Lipschitz. In this case, we can contruct the viscosity solution pair $(\varphi,b)$ as in \cite{SuiSun2022}, without using the stability estimate.

\medskip

\noindent
{\bf Acknowledgements}\quad
The author wish to thank Chengjian Yao, Bo Guan  for discussions and comments.
The first author is supported by National Natural Science Foundation of China (No. 12001138).
The second author is supported by
National Natural Science Foundation of China (No. 12371207)
and
a start-up grant from ShanghaiTech University.


\medskip



\begin{thebibliography}{999}


\bibitem{CNSIII}
Caffarelli, L. A., Nirenberg, L.,  Spruck, J. :
{The Dirichlet problem for nonlinear second-order elliptic equations III: Functions of eigenvalues of the Hessians}.
Acta Math. {\bf155}, 261--301 (1985).




\bibitem{Chen2015}
Chen, C.-Q. :
{The interior gradient estimate of Hessian quotient equations}.
J. Differ. Equ. {\bf 259}, 1014--1023 (2015).






\bibitem{Chen2000}
X.-X. Chen,
{\em The space of K\"ahler metrics}.
J. Differ. Geom. {\bf 56} (2000), 189--234.




\bibitem{CW1}
Chou, K.-S., Wang, X.-J. :
{A variational theory of the Hessian equation}.
Comm. Pure Appl. Math. {\bf 54} (9), 1029--1064 (2001).



\bibitem{Guan2014}
B. Guan,
{\em Second-order estimates and regularity for fully nonlinear elliptic equations on Riemannian manifolds}.
Duke Math. J., {\bf 163} (2014), 1491--1524.


\bibitem{Guan2023}
B. Guan,
{\em The Dirichlet problem for fully nonlinear elliptic equations on Riemaniian manifolds}.
Adv. Math. {\bf 415} (2023), 108899.




\bibitem{GP2022}
B. Guo and D. Phong,
{\em On $L^\infty$ estimates for fully nonlinear partial differential equations on Hermitian manifolds}.
arXiv:2204.12549.





\bibitem{GPT2021}
B. Guo, D. Phong and F. Tong,
{\em On $L^{\infty}$ estimates for complex Monge-Amp\`ere equations}.
To appear in Ann. Math..




\bibitem{Kolodziej2005}
S. Kolodziej,
{\em The complex Monge-Amp`ere equation and pluripotential theory}.
Mem. Amer. Math. Soc. {\bf 178} (2005).



\bibitem{Li1990}
Y.Y. Li,
{\em Some existence results for fully nonlinear elliptic equations of Monge-Amp\`ere type}.
Comm. Pure Appl. Math. {\bf 43} (1990), 233--271.



\bibitem{Lions1983}
P. L. Lions,
{\em Optimal control of diffusion processes and Hamilton-Jacobi-Bellman equations part 2: viscosity solutions and uniqueness}.
Comm. Part. Diff. Eq. {\bf 8}  (1983), 1229--1276.



\bibitem{SuiSun2022}
Z. Sui and W. Sun,
{\em Lipschitz continuous hypersurfaces with prescribed curvature and asymptotic boundary in hyperbolic space}.
Int. Math. Res. Notices {\bf 2022} (2022), no. 24, 19175--19221.



\bibitem{SuiSun20}
Z. Sui and W. Sun,
{\em Interior Gradient Estimates for General Prescribed Curvature Equations}.
Anal. Theory Appl. {\bf 39} (2023), 260--286.



\bibitem{SuiSun2023}
Z. Sui and W. Sun,
{\em On $L^\infty$ estimate for complex Hessian quotient equations on compact K\"ahler manifolds}.
J. Geom. Anal. {\bf 33} (2023) ,  165.




\bibitem{Sun202305MA}
W. Sun,
{\em The boundary case for complex Monge-Amp\`ere type equations}.
arXiv:2305.02576.



\bibitem{Sun202305}
W. Sun,
{\em The weak solutions to complex Hessian equations}.
To appear in Calc. Var. PDE.
















\bibitem{Szekelyhidi}
G. Sz\'ekelyhidi,
{ \em Fully non-linear elliptic equations on compact Hermitian manifolds}.
J. Differential Geom.  {\bf 109} (2018), 337--378.



\bibitem{Urbas}
J. Urbas,
{ \em Hessian equations on compact Riemannian manifolds}. Nonlinear Problems in Mathematical Physics and Related Topics II, Int. Math. Ser. (N. Y.) {\bf 2}, Kluwer/Plenum, New York, 2002, 367--377.


\bibitem{Wang2009}
Wang X.-J. :
{The $k$-Hessian equation}.
Geometric Analysis and PDEs, in: Lecture Notes in Math. {\bf 1977}, 177--252 (2009). Springer, Dordrecht.




\bibitem{Yuan2022}
R. Yuan,
{\em On the regularity of Dirichlet problem for fully non-linear elliptic equations on Hermitian manifolds}.
arXiv:2203.04898.







\end{thebibliography}
\end{document}